\documentclass[twoside]{article}
\usepackage[cp1251]{inputenc}
\usepackage[T2A]{fontenc}
\usepackage{array}
\usepackage{amssymb}
\usepackage{amsmath}
\usepackage{amsthm}
\usepackage{latexsym}
\usepackage{indentfirst}
\usepackage{bm}
\usepackage{enumerate}
\usepackage[dvips]{graphicx}
\usepackage{epsf}
\usepackage{wrapfig, hyperref}
\usepackage[mathscr]{euscript}
\usepackage{indentfirst}
\usepackage[english,russian]{babel}
\usepackage{textcomp}
\newtheorem{theorem}{Theorem}
\newtheorem{lemma}{Lemma}
\newtheorem{definition}{Definition}

\newtheorem{corollary}{Corollary}
\begin{document}
{\selectlanguage{english}
\binoppenalty = 10000 %
\relpenalty   = 10000 %

\pagestyle{headings} \makeatletter
\renewcommand{\@evenhead}{\raisebox{0pt}[\headheight][0pt]{\vbox{\hbox to\textwidth{\thepage\hfill \strut {\small Grigory K. Olkhovikov \& Guillermo Badia}}\hrule}}}
\renewcommand{\@oddhead}{\raisebox{0pt}[\headheight][0pt]{\vbox{\hbox to\textwidth{{Craig Interpolation Theorem fails  in Bi-intuitionistic Predicate Logic}\hfill \strut\thepage}\hrule}}}
\makeatother

\title{Craig Interpolation Theorem fails \\ in Bi-intuitionistic Predicate Logic}
\author{Grigory K. Olkhovikov\\ Department of Philosophy I\\ Ruhr University Bochum\\ Bochum, Germany \\
email: grigory.olkhovikov@\{rub.de, gmail.com\} \\ \\ Guillermo Badia\\ School of Historical and Philosophical Inquiry\\ University of Queensland\\ Brisbane, Australia \\
email: g.badia@uq.edu.au
\\ \\
\emph{This article is dedicated to the memory of Grigori Mints (1939-2014)}
}

\date{}
\maketitle
\begin{quote}
{\bf Abstract.} In this article we show that  bi-intuitionistic predicate logic lacks the Craig Interpolation Property. We proceed by adapting the counterexample given by Mints, Olkhovikov and Urquhart for intuitionistic predicate logic with constant domains \cite{mou}. More precisely, we show that there is a valid implication $\phi \rightarrow \psi$ with no interpolant (i.e. a formula $\theta$ in the intersection of the vocabularies of $\phi$ and $\psi$ such that both $\phi \rightarrow \theta$ and $\theta \rightarrow \psi$ are valid). Importantly, this result does not contradict the unfortunately named `Craig interpolation' theorem established by Rauszer in \cite{ra5} since that article is about the property more correctly named `deductive interpolation' (see Galatos, Jipsen, Kowalski and Ono's use of this term in  \cite{Gal}) for global consequence. Given that the deduction theorem fails for  bi-intuitionistic logic with global consequence, the two formulations of the property are not equivalent.
\end{quote}
\begin{quote}{\bf Keywords.} bi-intuitionistic predicate logic, Craig Interpolation Theorem, bi-asimulation
\end{quote}
\section{Introduction}

Bi-intuitionistic logic (or \emph{Heyting-Brouwer} logic) comes rather naturally from adding the algebraic dual of the intuitionistic implication $\rightarrow$ to the language of intuitionistic logic. This  connective is known as `co-implication', and  we denote it in what follows by `$\ll $'.\footnote{This is sometimes also called \emph{subtraction} \cite{re}. In the Kripke semantics for bi-intuitionistic logic, the new connective behaves similarly to a backwards looking diamond modality $\Diamond^{-1}$.}
In the 1970s, C. Rauszer started an intense study  of various technical aspects of both propositional and predicate bi-intuitionistic logic in a series of interesting articles spanning over a decade \cite{ra1, ra2, ra3, ra4, ra5, ra6}. 
This work has been picked up in  recent years by a number of scholars  \cite{ba1, ba2, Olk2,  go, gore2, gr, ko,  pi, sk, tr, cro} and results on all sorts of proof-theoretic and model-theoretic properties of these systems have been produced.

In this paper we are  concerned with refuting the Craig Interpolation Property for predicate bi-intuitionistic logic. In other words, we will show that there are valid implications $\phi \rightarrow \psi$ with no interpolant in the sense of  a formula $\theta$ in the intersection of the vocabularies of $\phi$ and $\psi$ such that both $\phi \rightarrow \theta$ and $\theta \rightarrow \psi$ are valid. Our counterexample is extracted from \cite{mou} where the Craig interpolation problem for predicate intuitionistic logic with constant domains is solved negatively. More specifically, we consider the implication $\phi \rightarrow \psi$ where  $$
\phi:= \forall x\exists y(P(y) \wedge (Q(y) \to R(x))) \wedge \neg\forall xR(x),
$$
and
$$
\psi:= \forall x(P(x) \to (Q(x) \vee S)) \to S.
$$
One of the reasons why our construction works is because  bi-intuitionistic predicate logic is well-known to be complete with respect to its constant domain semantics \cite{ra4}. Hence, the effect of adding the co-implication connective is that one can restrict attention to  Kripke models with constant domains (this has led some authors to question the  status of $\ll$ as an intuitionistic connective \cite{lop}).

C. Rauszer had already studied the interpolation problem for both predicate and propositional versions of bi-intuitionistic logic for global consequence in \cite{ra5}.\footnote{Later in the paper we will work with a local version of consequence. The many subtleties between the two notions of consequence in the bi-intuitionistic setting are thoroughly studied in \cite{gore2}. That paper also corrects several mistakes by Rauszer in her original work.} However, she had only established interpolation in a much weaker version properly known as `deductive interpolation' (see \cite{Gal}): if $\phi \vDash^g \psi$ there is  a formula $\theta$ in the intersection of the vocabularies of $\phi$ and $\psi$ such that both $\phi \vDash^g \theta$ and $\theta \vDash^g \psi$. Given that in bi-intuitionistic logic the deduction theorem fails for global consequence,\footnote{One counterexample, already in bi-intuitionistic propositional logic, is that $p \vDash^g (\top \ll p) \rightarrow \bot$, but $p\rightarrow ((\top \ll p) \rightarrow \bot)$ is not valid (see \cite{gore2}). } deductive interpolation is not equivalent to the usual Craig interpolation. Craig interpolation for propositional bi-intuitionistic logic was established only   recently by a complex proof-theoretic argument  \cite{ko}. Unfortunately, as we will see, those techniques cannot be extended to  the predicate case.

E. G. K. L\'opez-Escobar had observed in \cite{lop} that  bi-intuitionistic predicate logic is not conservative as an extension of  intuitionistic predicate logic simply because the axiom of constant domains is derivable in the former. It is known, however, that  bi-intuitionistic predicate logic is  a conservative extension of constant domain  intuitionistic predicate logic \cite{cro}. Then, one could wonder whether this trivialises the result in the present paper given the argument refuting interpolation for the latter in \cite{mou}. A minute of reflexion suffices to show the reader that this is not so, for the argument in \cite{mou} only establishes that the interpolant does not exist in the language of predicate intuitionistic logic which does not include the new co-implication connective. It does, on the other hand, indicate that the present result is quite natural.

Recently \cite{sh}, a number of questions have arisen as to the correctness of Rauszer's completeness  theorem for predicate bi-intuitionistic logic with respect to the Kripke semantics of constant domains. Three fundamental flaws of Rauszer's proofs are outlined  in \cite{sh}: (i) the fact that they mistakenly mix up properties of global and local consequence relations, (ii) the contradictory choice of rooted frames in the completeness proof which would validate non-theorems of bi-intuitionistic logic, and (iii) an incorrect application of some results by Gabbay. Incidentally, there is an alternative completeness argument  by Klemke in  \cite{kl},  where bi-intuitionistic predicate logic is studied possibly for the first time in print (and, as far as we know,  independently from Rauszer's work) and that contains other errors. This paper appears to have been sadly neglected in the history of the subject and it contains valuable techniques such as a version of the unraveling construction used in \cite{Olk2} (see \cite[Def. 4.1]{kl}). 

Let us note, however, that the central result in the present paper does not actually require any appeal to completeness. All we need is (i) that  the axiom of constant domains is a theorem of bi-intuitionistic predicate logic and (ii) that  bi-intuitionistic predicate logic is sound with respect to the Kripke semantics for intuitionistic logic of constant domains. This is because our implication  $\phi \rightarrow \psi$ described above is a theorem of intuitionistic predicate logic with the addition of the axiom of constant domains, and hence by (i), a theorem of bi-intuitionistic predicate logic. Furthermore, if there would be a proof-theoretic Craig interpolant for this implication, by (ii), it would be a semantic interpolant and hence we would get the desired contradiction reasoning as in Theorem \ref{T:main} below.

The article is arranged as follows: in \S \ref{1} we introduce all the required technical preliminaries that we will use in the construction of our counterexample, namely the language of the logic and its models, as well as  the notion of a bi-asimulation (among a few other things); in \S \ref{2} we introduce the Craig Interpolation Property, Mints's counterexample from \cite{mou} and we show how to adapt the technical argument in that paper to obtain the central result of the present article (Theorem \ref{T:main}); finally, in \S \ref{con} we provide some concluding remarks on the failure of interpolation described here.

\section{Preliminaries and notation}\label{1}

\subsection{The first-order language}
We start by fixing some general notational conventions. In this paper, we identify the natural numbers with finite ordinals. We denote by $\omega$ the smallest infinite ordinal, and by $\mathbb{N}$ the set $\omega \setminus \{0\}$. For any $n \in \omega$, we will denote by $\bar{o}_n$ the sequence
$(o_1,\ldots, o_n)$ of objects of any kind; moreover, somewhat
abusing the notation, we will denote $\{o_1,\ldots,
o_n\}$ by $\{\bar{o}_n\}$. The ordered $1$-tuple will
be identified with its only member. For any given $m,n \in \omega$, the notation $(\bar{p}_m)^\frown(\bar{q}_n)$ denotes the concatenation of $\bar{p}_m$ and $\bar{q}_n$, and the notation $\bar{p}_n\mapsto\bar{q}_n$ denotes the relation $\{(p_i, q_i)\mid 1 \leq i \leq n\}$; the latter relation is often explicitly assumed to define a function.

  More generally, if $f$ is any function, then we will denote by $dom(f)$ its domain
and by $rang(f)$ the image of $dom(f)$ under $f$; if $rang(f)
\subseteq M$, we will also write $f: dom(f) \to M$. If $f: X \to Y$ is any function, and $Z$ is any set, then we denote by $[f]Z$ the set $\{q \in Y\mid (\exists p \in Z)(f(p) = q)\}$; in particular, if $\bar{p}_n\mapsto\bar{q}_n$ defines a function, then $[\bar{p}_n\mapsto\bar{q}_n]Z$ stands for $\{q_i\mid p_i \in Z,\,1 \leq i \leq n\}$.

For a given set $\Omega$ and a $k\in \omega$, the notation $\Omega^k$
(resp. $\Omega^{\neq k}$) will denote the $k$-th Cartesian power
of $\Omega$ (resp. the set of all $k$-tuples from $\Omega^k$ such
that their elements are pairwise distinct). We remind the reader that in the special case when $k = 0$ it is usual to define $\Omega^0:= \{\emptyset\} = 1$, given our earlier convention about the natural numbers. Moreover we will denote the powerset of $\Omega$ by $\mathscr{P}(\Omega)$. We also define that $\Omega^\ast:= \bigcup_{n \geq 0}\Omega^n$. Finally, the notation $\mid X\mid$ will denote the power of the set $X$, so that, for example, $\mid X\mid = \omega$ will mean that $X$ is countably infinite.

In this paper, we consider a first-order language without equality based on any set of
predicate letters (including $0$-ary predicates, that is to say, propositional letters) and individual constants. We do not
allow functions, though.

An ordered couple of sets $\Sigma = (Pred_\Sigma, Const_\Sigma)$
comprising all the predicate letters and constants allowed in a
given version of the first-order language will be called the
\emph{signature} of this language. Signatures will be denoted by
letters $\Sigma$ and $\Theta$. For a given signature $\Sigma$, the
elements of $Pred_\Sigma$ will be denoted by $P^n$ and $Q^n$,
where $n
> 0$ indicates the arity, and the elements of $Const_\Sigma$ will
be denoted by lowercase Latin letters like $a$, $b$, $c$, $d$ and so on. All these notations and all of
the other notations introduced in this section can be decorated by
all sorts of sub- and superscripts. We will often use the notation
$\Sigma_n$ for the set of $n$-ary predicates in a given signature
$\Sigma$.

Even though we have defined signatures as ordered pairs, we will
somewhat abuse the notation in the interest of suggestivity and,
given a set $\Pi$ of predicates and a set $\Delta$ of individual
constants outside a given signature $\Sigma$, we will denote by $\Sigma \cup \Pi \cup \Delta$ the
signature where the predicates from $\Pi$ are added to
$Pred_\Sigma$ with their respective arities and the constants from
$\Delta$ are added to $Const_\Sigma$. Moreover, we will write $\Sigma \subseteq \Theta$ iff there exists a set $\Pi$ of predicates and a set $\Delta$ of individual
constants outside $\Sigma$ such that $\Theta = \Sigma \cup \Pi \cup \Delta$; in other words, iff $\Theta$ extends $\Sigma$ as a signature. In case $\Pi = Pred_{\Sigma'}$ and $\Delta = Const_{\Sigma'}$, we also express this same fact by writing $\Theta = \Sigma \cup \Sigma'$; furthermore, in case we have $P \in \Pi$ and $c \in \Delta$, we will write $P \notin \Sigma$ and $c \notin \Sigma$. Similarly, if $\Theta_1 = \Sigma \cup \Pi_1 \cup \Delta_1$ and $\Theta_2 = \Sigma \cup \Pi_2 \cup \Delta_1$, for pairwise non-overlapping sets of predicates $\Pi_1$ and $\Pi_2$ and sets of constants $\Delta_1$ and $\Delta_2$, we will write that $\Sigma = \Theta_1 \cap \Theta_2$. 

If $\Sigma$ is a signature, then the set of first-order formulas is generated
from $\Sigma$ in the usual way, using the set of logical symbols $\{ \bot, \top, \wedge, \vee, \to, \ll, \forall, \exists \}$ (here $\ll$ stands for the bi-intuitionistic co-implication) and the set of (individual)
variables $Var := \{ v_i \mid i < \omega \}$, and will be denoted by
$L(\Sigma)$. The elements of $Var$ will be also denoted by $x,y,z,
w$, and the elements of $L(\Sigma)$ by Greek letters like
$\phi$, $\psi$ and $\theta$. As is usual, we will use $\neg\phi$ as an abbreviation for $\phi \to \bot$.

For any given signature $\Sigma$, and any given $\phi \in
L(\Sigma)$, we define $BV(\phi)$ and $FV(\phi)$, its sets
of free and bound variables, by the usual inductions. These sets
are always finite. We will denote the set of $L(\Sigma)$-formulas with free
variables among the elements of $\bar{x}_n \in Var^{\neq n}$ by
$L_{\bar{x}_n}(\Sigma)$; in particular, $L_\emptyset(\Sigma)$
will stand for the set of $\Sigma$-sentences. If $\varphi \in
L_{\bar{x}_n}(\Sigma)$ (resp. $\Gamma \subseteq
L_{\bar{x}_n}(\Sigma)$), then we will also express this by
writing $\varphi(\bar{x}_n)$ (resp. $\Gamma(\bar{x}_n)$).

Similarly, given a $\phi \in
L(\Sigma)$, one can define a signature $\Theta_\phi$ such that, for any signature $\Sigma'$ we have $\phi \in L(\Sigma')$ iff $\Theta_\phi \subseteq \Sigma'$. The definition proceeds by induction on the construction of $\phi$ and looks as follows:
\begin{itemize}
	\item $\Theta_{P(\bar{t}_n)} = (\{P^n\}, \{\bar{t}_n\}\cap Const_\Sigma)$, for any $P \in \Sigma_n$ and $\bar{t}_n \in (Var \cup Const_\Sigma)^n$.
	
	\item $\Theta_\bot = \Theta_\top = \emptyset$.
	
	\item $\Theta_{\psi\circ\theta} = \Theta_{\psi}\cup \Theta_{\theta}$ for $\circ \in \{\wedge, \vee, \to, \ll\}$.
	
	\item $\Theta_{\circ\psi} = \Theta_{\psi}$ for $\circ \in \{\forall x, \exists x\mid x \in Var\}$.
\end{itemize}

If $X \subseteq Var$ is finite and $f: X \to Const_\Sigma$, then, for any $\phi \in L(\Sigma)$, we denote by $\phi[f] \in L(\Sigma)$ the result of simultaneously replacing every free occurrence of every $x \in X$ by $f(x)$. The precise definition of this operation proceeds by induction on the construction of $\varphi \in L(\Sigma)$ and runs as follows:
\begin{itemize}
	\item $P(\bar{t}_n) := P(\bar{s}_n)$, where $P \in \Sigma_n$, and $\bar{t}_n, \bar{s}_n \in (Var \cup Const_\Sigma)^n$ are such that, for all $1 \leq i \leq n$ we have:
	$$
	s_i := \begin{cases}
		f(t_i),\text{ if }t_i \in X;\\
		t_i,\text{ otherwise.}
	\end{cases}
	$$
	\item $(\phi)[f] : =  \phi$ for $\phi \in \{\bot, \top\}$.
	
	\item $(\phi \circ \psi)[f] : = \phi[f]\circ\psi[f]$, for $\circ \in \{\wedge, \vee, \to, \ll\}$.
	
	\item $(Qx\phi)[f] : = Qx(\phi[f\upharpoonright(X \setminus \{x\})])$, for $x \in Var$ and $Q \in \{\forall, \exists\}$.
\end{itemize}
Since $dom(f)$ is always finite, we will write $\phi[c_1/x_1,\ldots,c_n/x_n]$ (or, alternatively, $\phi[\bar{c}_n/\bar{x}_n]$) in place of $\varphi[f]$, whenever $f$ is $\bar{x}_n \mapsto \bar{c}_n$. It is clear that every $\varphi[f]$ can be written in this format. An important particular case is when $dom(f) = \{x\}$, so that we can write $\phi[f]$ as $\phi[c/x]$ for the corresponding $c \in Const_\Sigma$.

The following lemma states that our substitution operations work as expected. We omit the straightforward but tedious inductive proof.
\begin{lemma}\label{L:subst-basic}
Let $\Sigma$ be a signature, let $\phi \in L_{\bar{x}_n}(\Sigma)$, let $\bar{x}_n \in Var^{\neq n}$, $\bar{y}_m \in Var^{\neq m}$, and $\bar{z}_k \in Var^{\neq k}$	be such that $\{\bar{z}_k\} = \{\bar{x}_n\}\setminus\{\bar{y}_m\}$. Moreover, let $\bar{c}_m \in (Const_\Sigma)^m$, and let $(i_1,\ldots,i_m)$ be a permutation of $(1,\ldots, m)$. Then the following statements hold:
\begin{enumerate}
	\item $\phi[\bar{c}_m/\bar{y}_m] \in L_{\bar{z}_k}(\Sigma)$.
	
	\item $\phi[\bar{c}_m/\bar{y}_m] = \phi[c_{i_1}/y_{i_1},\ldots, c_{i_m}/y_{i_m}]$.
	
	\item  $\phi[\bar{c}_m/\bar{y}_m] = \phi[c_{1}/y_{1}]\ldots[c_{m}/y_{m}]$.
	
	\item If $FV(\phi) \cap \{\bar{x}_n\} = \{x_{j_1},\ldots,x_{j_m}\}$, then $\phi[\bar{c}_n/\bar{x}_n] = \phi[c_{j_1}/x_{j_1},\ldots,c_{j_m}/x_{j_m}]$.
	
	\item If $\psi \in  L_{\bar{x}_n}(\Sigma \cup \{\bar{c}_m\})$, and $\{\bar{y}_m\} \cap (\{\bar{x}_n\} \cup BV(\psi)) = \emptyset$, then there exists a $\chi \in L_{(\bar{x}_n)^\frown(\bar{y}_m)}(\Sigma)$ such that $\chi[\bar{c}_m/\bar{y}_m] = \psi$.
	
	\item $\Theta_{\phi[\bar{c}_m/\bar{y}_m]} \subseteq \Theta_\phi \cup \{\bar{c}_m\}$.
\end{enumerate}
\end{lemma} 
Note that the second part of the lemma holds just by the notational convention, since permutations of the set of the ordered pairs define one and the same function. Moreover, the joint effect of the second and the third part is that one can break up $[\bar{c}_n/\bar{x}_n]$ into a finite set of parts of arbitrary size and then apply those parts to a given $\phi$ in arbitrary order without affecting the result of a substitution. The latter is something that does not hold for the substitutions of arbitrary terms but is valid in our case due to the restriction to constants.

The notion of substitution is necessary for a correct inductive definition of a sentence that is independent from the inductive definition of an arbitrary formula.
More precisely, let $\Sigma$ be a signature and let $c$ be any constant, perhaps outside $\Sigma$. Then $L_\emptyset(\Sigma)$ is the smallest subset $Sent(\Sigma)$ of $L(\Sigma)$ satisfying the following conditions:
\begin{itemize}
	\item $P(\bar{c}_n), \bot, \top \in Sent(\Sigma)$ for all $n \geq 1$, $P \in \Sigma_n$, and $\bar{c}_n \in (Const_\Sigma)^n$.
	
	\item If $\phi, \psi \in Sent(\Sigma)$, then $(\phi\circ\psi) \in Sent(\Sigma)$ for all $\circ \in \{\wedge, \vee, \to, \ll\}$.
	
	\item If $x \in Var$ and $\phi[c/x] \in Sent(\Sigma\cup \{c\})$, then $\forall x\phi, \exists x\phi \in Sent(\Sigma)$.
\end{itemize}

\subsection{Semantics}

For any given signature $\Sigma$, a constant domain intuitionistic Kripke $\Sigma$-model is a structure of the form $\mathcal{M} = (W, \prec, D, V, I)$ such that:
\begin{itemize}
	\item $W$ is a non-empty set of worlds, or nodes.
	
	\item The accessibility relation $\prec \subseteq W \times W$ is reflexive and transitive (i.e. a pre-order).
	
	
	\item $D$ is a non-empty set (or domain) of objects which is disjoint from $W$.
	
	\item The mapping $V$ is a function from $Pred_\Sigma\times W$ into the set $\bigcup_{n \geq 0}\mathscr{P}(D^n)$ such that, for every $n \geq 0$, every $P \in \Sigma_n$, and all $w, v \in W$, it is true that $V(P, w) \subseteq D^n$ and:
	$$
	w \prec v \Rightarrow V(P, w) \subseteq V(P, v).
	$$
		
	Given a $P \in \Sigma_n$, we will sometimes consider the unary projection functions of the form $V_P:W \to \mathscr{P}(D^n)$ arising from the binary function $V$. It is clear that we may view $V$ as the union of the corresponding family of the unary functions, that is to say, that we can assume $V = \bigcup_{P \in Pred_\Sigma}V_P$.
	
	\item $I: Const_\Sigma\to D$ is the function interpreting the individual constants.
\end{itemize}
In what follows, we will always assume the particular case when $I$ is just the identity function on $Const_\Sigma$; in other words, we will only consider the constants that are names of themselves. 

An interesting particular case arises when $P \in \Sigma_0$. In this case, our definition says that $V_P:W \to \mathscr{P}(D^0)$, but $\mathscr{P}(D^0) = \mathscr{P}(\{\emptyset\}) = \{\emptyset, \{\emptyset\}\} = \{0,1\} = 2$, given our identification of natural numbers with the finite ordinals. Therefore, $V_{P}$ in this case is, in effect, a function from $W$ to $\{0, 1\}$, and $V_{P}$ returns $1$ for a given $w$ iff $P^0$ holds at $w$ in $\mathcal{M}$.

Since we will only consider in this paper the models of the type described above, we will simply call them $\Sigma$-models.

When we use subscripts and other decorated model notations, we strive for consistency in this respect. Some examples of this notational principle are given below:
$$
\mathcal{M} = (W, \prec, D, V, I),\,
\mathcal{M}' = (W', \prec', D', V', I'),\, \mathcal{M}_n = (W_n, \prec_n, D_n, V_n, I_n)
$$
As is usual, we denote the reduct of a $\Sigma$-model
$\mathcal{M}$ to a signature $\Theta \subseteq \Sigma$ by
$\mathcal{M}\upharpoonright\Theta$.

As for the reverse operation of expanding a model of a smaller signature to a model of a larger signature, if $\mathcal{M}$ is a $\Sigma$-model, $P^{n_1}_1, \ldots P^{n_k}_k \notin \Sigma$ are pairwise distinct and $\mathbf{A} \subseteq D$, then, for any given sequence of $\pi_1:W\to \mathscr{P}(D^{n_1}), \ldots, \pi_k:W\to \mathscr{P}(D^{n_k})$ of functions monotonic relative to $\prec$, we will denote by $(\mathcal{M}; \bar{P}_k\mapsto \bar{\pi}_k; \mathbf{A})$ the unique $\Sigma \cup \{ P^{n_1}_1, \ldots P^{n_k}_k\}\cup \mathbf{A}$-model $\mathcal{M}'$ such that $\mathcal{M}'\upharpoonright\Sigma = \mathcal{M}$, $V'(P^{n_i}_i, w) = \pi_i(w)$ for all $w \in W$, and $I'(a) = a$ for all $1 \leq i \leq k$ and all $a \in \mathbf{A}$. In case $k = 0$ or $\mathbf{A} = \emptyset$, we will write $(\mathcal{M};\mathbf{A})$ or $(\mathcal{M}; \bar{P}_k\mapsto \bar{\pi}_k)$, respectively.

The semantics is given by the following forcing relation defined by induction on the construction of a sentence. If $\Sigma$ is a signature, $\mathcal{M}$ is a $\Sigma$-model, $w \in W$, and $\phi \in L_\emptyset(\Sigma \cup D)$ then we write that $(\mathcal{M}; D), w \models \phi$ and say that $\phi$ is true at $w$ in $(\mathcal{M}; D)$ iff it follows from the following clauses:
\begin{align*}
	(\mathcal{M}; D), w\models P(\bar{a}_n) &\Leftrightarrow \bar{a}_n \in V_P(w) &&P\in \Sigma_n,\,\bar{a}_n\in D^n\\
	(\mathcal{M}; D), w\models \phi \wedge \psi &\Leftrightarrow (\mathcal{M}; D), w\models \phi\text{ and }(\mathcal{M}; D), w\models \psi\\
	\mathcal{M}, w\models \phi \vee \psi &\Leftrightarrow (\mathcal{M}; D), w\models \phi\text{ or }(\mathcal{M}; D), w\models \psi\\
	(\mathcal{M}; D), w\models \phi \to \psi &\Leftrightarrow (\forall v\succ w)((\mathcal{M}; D), v\not\models \phi\text{ or }(\mathcal{M}; D), w\models \psi)\\
	(\mathcal{M}; D), w\models \phi \ll \psi &\Leftrightarrow (\exists v\prec w)((\mathcal{M}; D), v\models \phi\text{ and }(\mathcal{M}; D), w\not\models \psi)\\ 
	(\mathcal{M}; D), w\models \forall x\phi &\Leftrightarrow (\forall a\in D)((\mathcal{M}; D), w\models \phi[a/x])\\
	(\mathcal{M}; D), w\models \exists x\phi &\Leftrightarrow (\exists a\in D)((\mathcal{M}; D), w\models \phi[a/x])	   
\end{align*}
Note that the assumption of constant domains allows for a simplification of the clause treating the universal quantifier. The following lemma then spells out the consequences of our definition for the defined connective $\neg$:
\begin{lemma}\label{L:neg-and-iff}
	If $\Sigma$ is a signature, $\mathcal{M}$ is a $\Sigma$-model, $w \in W$, and $\phi \in L_\emptyset(\Sigma \cup D)$ then we have:
	\begin{align*}
	(\mathcal{M}; D), w\models \neg\phi &\Leftrightarrow (\forall v\succ w)((\mathcal{M}; D), v\not\models \phi)
	\end{align*}
\end{lemma}
The proof is immediate by the definition and is omitted.

We then extend this semantics to arbitrary models by setting that for a $\Sigma$-model $\mathcal{M}$, and for a $\phi \in L_\emptyset(\Sigma)$, we have $\mathcal{M}, w\models \phi$ iff $(\mathcal{M};D), w\models \phi$. The following theorem is then immediate:
\begin{theorem}\label{T:expansion-property}
Let $\Sigma, \Theta$ be signatures such that $\Sigma \subseteq \Theta$, let $\mathcal{M}$ be a $\Theta$-model, let $w \in W$, and let $\phi \in L_\emptyset(\Sigma)$. Then $\mathcal{M}, w \models \phi$ iff $\mathcal{M}\upharpoonright\Sigma, w \models \phi$. 	
\end{theorem}
Theorem \ref{T:expansion-property} is often referred to as Expansion Property in the literature on abstract model theory.

Next, we adopt the standard definitions of validity, satisfiability and of semantic consequence relation. In particular, given a signature $\Sigma$, a $\Gamma \cup \{\phi\} \subseteq L_\emptyset(\Sigma)$, a $\Sigma$-model $\mathcal{M}$, and a $w \in W$, we say that $(\mathcal{M}, w)$ \textit{satisfies} $\Gamma$, and write $\mathcal{M}, w\models \Gamma$ iff we have $\mathcal{M}, w\models \psi$ for all $\psi \in \Gamma$. Furthermore, we say that $\phi$ is a \textit{consequence} of $\Gamma$ and write $\Gamma \models \phi$ iff for every $\Sigma$-model $\mathcal{M}$, and every $w \in W$, $\mathcal{M}, w\models \Gamma$ implies that $\mathcal{M}, w\models \phi$. In case $\Gamma = \{\psi\}$ for some $\psi \in L_\emptyset(\Sigma)$, we will omit the brackets and simply write $\psi \models \phi$, and in case $\Gamma = \emptyset$ we will write $\models \phi$ omitting $\Gamma$ altogether.

Having now both the language and its semantic apparatus in place, we can speak of the logical system $\mathsf{QBIL}$ of  bi-intuitionistic predicate logic. As a logic, $\mathsf{QBIL}$ can be presented as a fragment of  classical first-order logic $\mathsf{QCL}$ by means of an appropriate standard translation. On the other hand, $\mathsf{QBIL}$ can be seen as resulting from  intuitionistic predicate logic $\mathsf{QIL}$ by first restricting its semantics to the constant-domain Kripke models and then adding $\ll$ as the new connective to the language. An even closer relation exists between $\mathsf{QBIL}$ and the intuitionistic logic of constant domains $\mathsf{CD}$ which can be viewed as the $\ll$-free fragment of $\mathsf{QBIL}$.

\subsection{Bi-asimulations}
An important semantical concept related to  bi-intuitionistic predicate logic is the notion of  first-order bi-asimulation. This notion can be defined as follows:
\begin{definition}\label{D:bi-asimulation}
Let $\Sigma$ be a signature, let $\mathcal{M}_0$, $\mathcal{M}_1$ be $\Sigma$-models. A non-empty relation $A$ is called a first-order bi-asimulation iff the following conditions are satisfied for all $i,j$ such that $\{i,j\} = \{0,1\}$, and for all $n \in \omega$, all $P \in \Sigma_m$, all (maybe non-distinct) $j_1, \ldots, j_m \leq n$, all $(w)^\frown(\bar{a}_n) \in W_i\times (D_i)^n$ and all $(v)^\frown(\bar{b}_n) \in W_j\times (D_j)^n$ such that $(w)^\frown(\bar{a}_n)\mathrel{A}(v)^\frown(\bar{b}_n)$:
\begin{align}
&A \subseteq \bigcup_{n \in \omega}((W_0\times (D_0)^n) \times (W_1\times (D_1)^n) )\cup ((W_1\times (D_1)^n) \times (W_0\times (D_0)^n) )\label{E:type}\tag{type}\\
 &(\mathcal{M}_i; D_i), w\models P(a_{j_1},\ldots,a_{j_m}) \Rightarrow (\mathcal{M}_j; D_j), v\models P(b_{j_1},\ldots,b_{j_m})\label{E:atom}\tag{atom}\\
&(\forall v_0 \succ_j v)(\exists w_0\succ_i w)((w_0)^\frown(\bar{a}_n)\mathrel{A}(v_0)^\frown(\bar{b}_n)\,\&\,(v_0)^\frown(\bar{b}_n)\mathrel{A}(w_0)^\frown(\bar{a}_n))\label{E:back}\tag{back}\\
&(\forall w_0 \prec_i w)(\exists v_0\prec_j v)((w_0)^\frown(\bar{a}_n)\mathrel{A}(v_0)^\frown(\bar{b}_n)\,\&\,(v_0)^\frown(\bar{b}_n)\mathrel{A}(w_0)^\frown(\bar{a}_n))\label{E:forth}\tag{forth}\\
&(\forall b \in D_j)(\exists a\in D_i)((w)^\frown(\bar{a}_n)^\frown(a)\mathrel{A}(v)^\frown(\bar{b}_n)^\frown(b))\label{E:left}\tag{left}\\
&(\forall a \in D_i)(\exists b\in D_j)((w)^\frown(\bar{a}_n)^\frown(a)\mathrel{A}(v)^\frown(\bar{b}_n)^\frown(b))\label{E:right}\tag{right}
\end{align}
\end{definition} 
In the case where we have $(w)^\frown(\bar{a}_n)\mathrel{A}(v)^\frown(\bar{b}_n)$ in the assumptions of Definition \ref{D:bi-asimulation}, we will say that the first-order bi-asimulation $A$ is from $(w)^\frown(\bar{a}_n)$ to $(v)^\frown(\bar{b}_n)$.

Since we are not going to discuss any other kinds of bi-asimulation relations in this paper, we will omit in what follows the qualification `first-order' and will simply speak of bi-asimulations. The following lemma shows that first-order bi-intuitionistic formulas are preserved under bi-asimulations:
\begin{lemma}\label{L:preserve}
	Let $\Sigma$ be a signature, let $\mathcal{M}_0$, $\mathcal{M}_1$ be $\Sigma$-models, let $n \in \omega$, let $i,j$ such that $\{i,j\} = \{0,1\}$, let $(w)^\frown(\bar{a}_n) \in W_i\times (D_i)^n$ and $(v)^\frown(\bar{b}_n) \in W_j\times (D_j)^n$, and let $A$ be a bi-asimulation between $\mathcal{M}_0$ and $\mathcal{M}_1$ such that $(w)^\frown(\bar{a}_n)\mathrel{A}(v)^\frown(\bar{b}_n)$. Then for every tuple $\bar{x}_n \in Var^{\neq n}$ and every $\phi \in L_{\bar{x}_n}(\Sigma)$ it is true that:
	$$
	(\mathcal{M}_i; D_i), w\models \phi[\bar{a}_n/\bar{x}_n] \Rightarrow (\mathcal{M}_j; D_j), v\models  \phi[\bar{b}_n/\bar{x}_n].
	$$
\end{lemma}
\begin{proof}
	By induction on the construction of $\phi$. 
	
	\textit{Induction basis}. If $\phi = P(x_{j_1},\ldots,x_{j_m})$ for some (maybe non-distinct) $j_1, \ldots, j_m \leq n$ and some $P \in \Sigma_m$, then the statement of the Lemma follows from condition \eqref{E:atom} of Definition \ref{D:bi-asimulation} and the fact that we have $\phi[\bar{a}_n/\bar{x}_n] = P(a_{j_1},\ldots,a_{j_m})$ and $\phi[\bar{b}_n/\bar{x}_n] = P(b_{j_1},\ldots,b_{j_m})$.
	
	\textit{Induction step}. The cases when $\phi = \psi \wedge \theta$ and when  $\phi = \psi \vee \theta$ are straightforward. We consider the remaining cases:
	
	\textit{Case 1}. $\phi = \psi \to \theta$. If $(\mathcal{M}_j; D_j), v\not\models (\psi \to \theta)[\bar{b}_n/\bar{x}_n]$, then $(\mathcal{M}_j; D_j), v\not\models \psi[\bar{b}_n/\bar{x}_n] \to \theta[\bar{b}_n/\bar{x}_n]$, hence there exists a $v_0 \mathrel{\succ_j}v$ such that both $(\mathcal{M}_j; D_j), v_0\models \psi[\bar{b}_n/\bar{x}_n]$ and $(\mathcal{M}_j; D_j), v_0\not\models \theta[\bar{b}_n/\bar{x}_n]$. But then, by condition \eqref{E:back}, there must be a $w_0 \mathrel{\succ_i}w$ such that both $(w_0)^\frown(\bar{a}_n)\mathrel{A}(v_0)^\frown(\bar{b}_n)$ and $(v_0)^\frown(\bar{b}_n)\mathrel{A}(w_0)^\frown(\bar{a}_n)$. Now the Induction Hypothesis implies that both
	$(\mathcal{M}_i; D_i), w_0\models \psi[\bar{a}_n/\bar{x}_n]$ and $(\mathcal{M}_i; D_i), w_0\not\models \theta[\bar{a}_n/\bar{x}_n]$. Thus we get that:
	$$
	(\mathcal{M}_i; D_i), w\not\models (\psi[\bar{a}_n/\bar{x}_n] \to \theta[\bar{a}_n/\bar{x}_n]) = (\psi \to \theta)[\bar{a}_n/\bar{x}_n].
	$$
	
	\textit{Case 2}. $\phi = \psi \ll \theta$. The case is dual to Case 1.
	
	\textit{Case 3}. $\phi = \forall x\psi$. Then we might have $x \in \{\bar{x}_n\}$, but we can always avoid this inconvenience by choosing an $m \geq 0$, and some pairwise distinct $j_1,\ldots, j_m \leq n$ such that $\{x_{j_1},\ldots,x_{j_m}\} = FV(\exists x\phi)\cap \{\bar{x}_n\}$. For this $m$-tuple, we will have, $x \notin \{x_{j_1},\ldots,x_{j_m}\}$, and Lemma \ref{L:subst-basic}.4 implies that:
	$$
	(\forall x\psi)[\bar{b}_n/\bar{x}_n] = (\forall x\psi)[b_{j_1}/x_{j_1},\ldots,b_{j_m}/x_{j_m}] = \forall x(\psi[b_{j_1}/x_{j_1},\ldots,b_{j_m}/x_{j_m}]).
	$$
	But then we reason as follows. If $(\mathcal{M}_j; D_j), v\not\models (\forall x\psi)[\bar{b}_n/\bar{x}_n] = \forall x(\psi[b_{j_1}/x_{j_1},\ldots,b_{j_m}/x_{j_m}])$, then we must have $(\mathcal{M}_j; D_j), v\not\models \psi[b_{j_1}/x_{j_1},\ldots,b_{j_m}/x_{j_m}][b/x]$ for some $b \in D_j$, hence $(\mathcal{M}_j; D_j), v\not\models \psi[b_{j_1}/x_{j_1},\ldots,b_{j_m}/x_{j_m}, b/x]$ by Lemma \ref{L:subst-basic}. Moreover, condition \eqref{E:left} of the Definition \ref{D:bi-asimulation} implies that, for some $a \in D_i$ we must have $(w)^\frown(\bar{a}_n)^\frown(a)\mathrel{A}(v)^\frown(\bar{b}_n)^\frown(b)$. Since also $FV(\psi)\subseteq \{x_{j_1},\ldots,x_{j_m}\} \cup \{x\}$, the Induction Hypothesis is applicable and yields that $(\mathcal{M}_i; D_i), w\not\models \psi[a_{j_1}/x_{j_1},\ldots,a_{j_m}/x_{j_m}, a/x]$, whence, further, that $(\mathcal{M}_i; D_i), w\not\models \psi[a_{j_1}/x_{j_1},\ldots,a_{j_m}/x_{j_m}][a/x]$ by Lemma \ref{L:subst-basic}. 
	
	But then it follows (again applying Lemma \ref{L:subst-basic}.4) that 
	$$
	(\mathcal{M}_i; D_i), w\not\models \forall x(\psi[a_{j_1}/x_{j_1},\ldots,a_{j_m}/x_{j_m}]) = (\forall x\psi)[\bar{a}_n/\bar{x}_n].
	$$
	
	\textit{Case 4}. $\phi = \exists x\psi$. The case is dual to Case 3.
\end{proof}

\textbf{Remark 1}. It is rather straightforward to show, by combining the arguments given in \cite{ba1} and \cite{o} that the invariance under bi-asimulations given in Definition \ref{D:bi-asimulation} defines, for any given signature $\Sigma$, the set of natural standard translations of first-order bi-intuitionistic formulas into classical first-order logic. Such a proof, however, is beyond the scope of the present paper.

\section{Interpolation fails in the  bi-intuitionistic setting}\label{2}

\subsection{Craig Interpolation Property and Mints's Counterexample}
The Craig Interpolation Property, which was initially defined for $\mathsf{QCL}$, can be considered for $\mathsf{QBIL}$ without any changes in the original definition. Therefore, we will say that $\mathsf{QBIL}$ has \textit{Craig Interpolation Property} (CIP) iff for any signature $\Sigma$, and any $(\phi \to \psi) \in L_\emptyset(\Sigma)$ such that $\models \phi \to \psi$, there exists a $\theta \in L_\emptyset(\Theta_\phi\cap \Theta_\psi)$ such that both  $\models \phi \to \theta$ and  $\models \theta \to \psi$. Any such $\theta$ is then called an \textit{interpolant} for $\phi \to \psi$.

Although CIP holds for both $\mathsf{QCL}$ and $\mathsf{QIL}$, it is known to fail for $\mathsf{CD}$. The counterexample to CIP for the latter logic was published in \cite{mou} and is obtained by setting
$$
\phi:= \forall x\exists y(P(y) \wedge (Q(y) \to R(x))) \wedge \neg\forall xR(x),
$$
and
$$
\psi:= \forall x(P(x) \to (Q(x) \vee S)) \to S
$$
in the definition for CIP given in the previous paragraph. Under these particular settings (fixed throughout the remaining part of this paper), we will call $\phi \to \psi$ \textit{Mints's Counterexample} to CIP. It was shown in \cite{mou} that $\phi \to \psi$ is valid in $\mathsf{CD}$, but for no $\ll$-free $\theta \in  L_\emptyset(\Theta_\phi\cap \Theta_\psi) = L_\emptyset(\{P^1, Q^1\})$ both $\phi \to \theta$ and  $\theta \to \psi$ are valid in $\mathsf{CD}$.

Since $\mathsf{CD}$ is exactly the $\ll$-free fragment of $\mathsf{QBIL}$, the arguments of \cite{mou} also show that $\models\phi \to \psi$. These arguments, however, are insufficient to show that $\phi \to \psi$ lacks an interpolant in $\mathsf{QBIL}$ due to the richer language of the latter logic. The main purpose of the present paper is to close this gap and to show that Mints's Counterexample works for $\mathsf{QBIL}$ as well, and that therefore $\mathsf{QBIL}$ lacks CIP. The next subsection contains a proof of these claims.  

\subsection{Refuting CIP}
Our proof is an improvement on the proof given in \cite{mou}; it uses many of the same ideas albeit their application to this case requires several adjustments and a slight generalization.

We start by setting $\Sigma := \{P^1, Q^1\}$. We will describe two particular $\Sigma$-models and a bi-asimulation relation between them, and then we will show how to extend them to models satisfying $\phi$ and failing $\psi$, respectively. The rest of the argument will then be exactly as in \cite{mou}.

We start with the definitions of the building blocks for our models:
\begin{definition}\label{D:quasi-partitions}
	\begin{enumerate}
		\item A {\em quasi-partition} $(A,B,C)$ is defined by the
		following conditions:
		\begin{enumerate}
			\item $A \cup B \cup C = \mathbb{N}$; 
			\item $A,B,C$ are pairwise
			disjoint; 
			\item $\mid A\mid = \mid C\mid = \omega$; 
			\item $\mid B\mid \in \{\emptyset, \omega\}$.
		\end{enumerate}
		\item A relation $\trianglelefteq$ on the set of all
		quasi-partitions is defined by
		\[
		(A,B,C) \trianglelefteq (D,E,F) \Leftrightarrow [ A \subseteq D \wedge F \subseteq C].
		\]
	\end{enumerate}
\end{definition}
We immediately fix the following corollary to Definition \ref{D:quasi-partitions}:
\begin{corollary}\label{C:inclusion}
	For any quasi-partitions $(A,B,C)$ and $(D,E,F)$, if $(A,B,C) \trianglelefteq (D,E,F)$, then $A \cup B \subseteq D \cup E$. 
\end{corollary}
\begin{proof}
	If $(A,B,C) \trianglelefteq (D,E,F)$, then $F \subseteq C$, in other words, $\mathbb{N}\setminus (D \cup E) \subseteq \mathbb{N}\setminus (A \cup B)$. Therefore, by contraposition,  $A \cup B \subseteq D \cup E$. 
\end{proof}

We now fix two special quasi-partitions $\mathbf{v} =
(\mathbf{v}_1, \mathbf{v}_2, \mathbf{v}_3) = ( 3 \mathbb{N}, 3
\mathbb{N} + 1, 3 \mathbb{N} + 2 )$ and $\mathbf{w} =
(\mathbf{w}_1, \mathbf{w}_2, \mathbf{w}_3) = ( 2 \mathbb{N},
\emptyset, 2 \mathbb{N} + 1)$. Our (fixed till the end of the present section) $\Sigma$-models $\mathcal{M}_1$ and $\mathcal{M}_2$ are then as follows:
\begin{definition}\label{D:M}
	The structures $\mathcal{M}_1$ and $\mathcal{M}_2$ are such that:
	\begin{enumerate}
		\item $W_1 = W_2 = W$ is the set of all quasi-partitions.

		
		\item $\prec_2 := \trianglelefteq$, and, for all $(A,B,C), (D,E,F)\in W$, we have $(A,B,C)\mathrel{\prec_1}(D,E,F)$ iff $(A,B,C)\mathrel{\trianglelefteq}(D,E,F)$ and 
		$$
		(\mathbf{v} \mathrel{\trianglelefteq} (A, B, C) \,\&\,  \mid B \cap \mathbf{v}_2\mid = \omega) \Rightarrow \mid E \cap \mathbf{v}_2\mid = \omega.
		$$
		
		\item $D_1 = D_2 = \mathbb{N}$.
		
		\item $V_i(P,(A,B,C)) = A \cup B$, $V_i(Q,(A,B,C)) = A$ for all $(A,B,C) \in W$, and $i \in \{1, 2\}$.
		
		\item $I_1 = I_2 = \emptyset$.
	\end{enumerate}
\end{definition}
We now show that $\mathcal{M}_1$, $\mathcal{M}_2$ are $\Sigma$-models that can be extended to models verifying $\phi$ and falsifying $\psi$, respectively.
\begin{lemma}\label{L:models}
$\mathcal{M}_1$ and $\mathcal{M}_2$ are $\Sigma$-models.
\end{lemma}
\begin{proof}
	The extensions of $Q$ are monotonic relative to $\trianglelefteq$ (hence also relative to $\prec_i$ for $i \in \{1,2\}$) by definition of $V_i$. The extensions of $P$ are monotonic relative to $\trianglelefteq$  (and thus also relative to $\prec_i$ for $i \in \{1,2\}$) by Corollary \ref{C:inclusion}.
	
	The relation $\prec_2 = \trianglelefteq$ is clearly a pre-order on $W$. We show that $\prec_1$ is a pre-order on the same set.
	
	\textit{Reflexivity}. Let $(A, B, C) \in W$. We have $(A,B,C)\mathrel{\prec_1}(A,B,C)$, since both $(A,B,C)\mathrel{\trianglelefteq}(A,B,C)$ and
$$
(\mathbf{v} \trianglelefteq (A, B, C)\, \&\,  \mid B \cap \mathbf{v}_2\mid = \omega) \Rightarrow \mid B \cap \mathbf{v}_2\mid = \omega
$$
	are trivially satisfied.
	
	\textit{Transitivity}. Let $(A, B, C), (D,E,F), (G,H,I) \in W$ be such that $(A,B,C)\mathrel{\prec_1}(D,E,F)$ and $(D,E,F)\mathrel{\prec_1}(G,H,I)$. Then we must have both $(A,B,C)\mathrel{\trianglelefteq}(D,E,F)$ and $(D,E,F)\mathrel{\trianglelefteq}(G,H,I)$, and thus  $(A,B,C)\mathrel{\trianglelefteq}(G,H,I)$.
	
	On the other hand, we must have both:
	\begin{equation}\label{E:tr1}
	(\mathbf{v}\mathrel{\trianglelefteq} (A, B, C)\, \&\,  \mid B \cap \mathbf{v}_2\mid = \omega) \Rightarrow \mid E \cap \mathbf{v}_2\mid = \omega
	\end{equation}
	and:
	\begin{equation}\label{E:tr2}
	(\mathbf{v} \mathrel{\trianglelefteq} (D,E,F)\, \&\,  \mid E \cap \mathbf{v}_2\mid = \omega) \Rightarrow \mid H \cap \mathbf{v}_2\mid = \omega
	\end{equation}
	But then we reason as follows:
	\begin{align}
	&\mathbf{v} \mathrel{\trianglelefteq} (A, B, C)\label{E:p1} &&\text{(premise)}\\
	&\mid B \cap \mathbf{v}_2\mid = \omega\label{E:p2} &&\text{(premise)}\\
	&\mid E \cap \mathbf{v}_2\mid = \omega\label{E:p3} &&\text{(by \eqref{E:tr1}, \eqref{E:p1} and \eqref{E:p2})}\\
	&\mathbf{v} \trianglelefteq (D,E,F)\label{E:p4} &&\text{(by \eqref{E:p1} and $(A,B,C)\mathrel{\trianglelefteq}(D,E,F)$)}\\
	&\mid H \cap \mathbf{v}_2\mid = \omega\label{E:p5} &&\text{(by \eqref{E:tr2}, \eqref{E:p3} and \eqref{E:p4})}
	\end{align}
	Thus we have shown that both $(A,B,C)\mathrel{\trianglelefteq}(G,H,I)$ and:
	$$
		(\mathbf{v} \trianglelefteq (A,B,C) \&  \mid B \cap \mathbf{v}_2\mid = \omega) \Rightarrow \mid H \cap \mathbf{v}_2\mid = \omega,
	$$
whence $(A,B,C)\mathrel{\prec_1}(G,H,I)$ also follows.	
\end{proof}
We pause to state another corollary we need towards our main lemma:
\begin{corollary}\label{C:M1}
For arbitrary $(A,B,C), (D,E,F) \in W$, the following statements hold:
\begin{enumerate}
	\item $\mathbf{v}\mathrel{\prec_1}(A,B,C)$ iff $\mathbf{v}\mathrel{\trianglelefteq}(A,B,C)$ and $\mid B \cap \mathbf{v}_2\mid = \omega$.
	
	\item If both $\mathbf{v}\mathrel{\prec_1}(A,B,C)$ and $\mathbf{v}\mathrel{\prec_1}(D,E,F)$, then 
	$$
	(A,B,C)\mathrel{\prec_1}(D,E,F) \Leftrightarrow (A,B,C)\mathrel{\trianglelefteq}(D,E,F).
	$$
	
	\item If both $(A,B,C)\mathrel{\trianglelefteq}(D,E,F)$ and $(D,E,F)\mathrel{\trianglelefteq}(A,B,C)$, then $(A,B,C) = (D,E,F)$.
\end{enumerate}	
\end{corollary}
\begin{proof}
	(Part 1). We have $\mathbf{v}\mathrel{\prec_1}(A,B,C)$ iff $\mathbf{v}\mathrel{\trianglelefteq}(A,B,C)$ and
	$$
	(\mathbf{v} \mathrel{\trianglelefteq} \mathbf{v} \,\&\,  \mid \mathbf{v}_2 \cap \mathbf{v}_2\mid = \omega) \Rightarrow \mid B \cap \mathbf{v}_2\mid = \omega,
	$$
	but, since the premise of the latter conditional is trivially true, it holds iff $\mid B \cap \mathbf{v}_2\mid = \omega$.
	
	(Part 2). By definition, $(A,B,C)\mathrel{\prec_1}(D,E,F)$ implies $(A,B,C)\mathrel{\trianglelefteq}(D,E,F)$. In the other direction, if $(A,B,C)\mathrel{\trianglelefteq}(D,E,F)$ and, additionally, $\mathbf{v}\mathrel{\prec_1}(D,E,F)$, then we must have, by Part 1, that $\mid E \cap \mathbf{v}_2\mid = \omega$. But then also the conditional 
	 $$
	 (\mathbf{v} \trianglelefteq \mathbf{v} \&  \mid \mathbf{v}_2 \cap B\mid = \omega) \Rightarrow \mid E \cap \mathbf{v}_2\mid = \omega
	 $$ 
	 must be trivially true, whence also $(A,B,C)\mathrel{\prec_1}(D,E,F)$ follows.
	 
	 (Part 3). If both $(A,B,C)\mathrel{\trianglelefteq}(D,E,F)$ and $(D,E,F)\mathrel{\trianglelefteq}(A,B,C)$ then we have both $A = D$ and $C = F$. Now $B = E$ follows by the fact that $(A,B,C), (D,E,F)$ are quasi-partitions.
\end{proof}

In case we have $(A,B,C)\mathrel{\trianglelefteq}(D,E,F)$ but $(A,B,C) \neq (D,E,F)$, we will write $(A,B,C)\mathrel{\triangleleft}(D,E,F)$. 
\begin{lemma}\label{L:extendable}
	Fix a surjective $f: \mathbb{N} \to \mathbf{v}_2$, and let $\sigma_1: W \to \mathscr{P}(\mathbb{N})$ and $\sigma_2: W \to \{0,1\}$ be defined as follows for all $(A,B,C) \in W$:
		$$
		\sigma_1(A,B,C):= \{n\mid f(n) \in A \}
		$$
		
			$$
		\sigma_2(A,B,C):= \begin{cases}
			1,\text{ if } \mathbf{w}_1\mathrel{\triangleleft}(A,B,C);\\
			0,\text{ otherwise. }
		\end{cases}
		$$
	Then $\sigma_1, \sigma_2$ are monotonic relative to $\trianglelefteq$, and hence also relative to $\prec_i$ for all $i \in \{1,2\}$.

	Furthermore, we define that $\mathcal{M}'_1 := (\mathcal{M}_1; R^1\mapsto \sigma_1)$ and $\mathcal{M}'_2 := (\mathcal{M}_2; S^0\mapsto \sigma_2)$.
	Then both $\mathcal{M}'_1,\mathbf{v}\models \phi$ and $\mathcal{M}'_2,\mathbf{w}\not\models \psi$.
\end{lemma}
\begin{proof}
We deal with the monotonicity claims first. As for $\sigma_1$, if $(A,B,C)\mathrel{\trianglelefteq}(D,E,F)$, and $n \in \sigma_1(A,B,C)$, then $f(n) \in A \subseteq D$, hence also $n \in \sigma_1(D,E,F)$. As for $\sigma_2$, if $(A,B,C)\mathrel{\trianglelefteq}(D,E,F)$ and  $\sigma_2(A,B,C) = 1$, then $\mathbf{w}_1\mathrel{\triangleleft}(A,B,C)$, hence also $\mathbf{w}_1\mathrel{\triangleleft}(D,E,F)$ and $\sigma_2(D,E,F) = 1$.

It remains to show the satisfaction claims for the extended models:

($\mathcal{M}'_1,\mathbf{v}\models \phi$). Assume that $(A,B,C) \in W$ is such that $\mathbf{v}\mathrel{\prec_1}(A,B,C)$. Then, by Corollary \ref{C:M1}.1, we have $\mathbf{v}\mathrel{\trianglelefteq}(A,B,C)$ and $\mid B \cap \mathbf{v}_2\mid = \omega$. So we choose any $n \in  B \cap \mathbf{v}_2$ and any $m \in \mathbb{N}$ such that $f(m) = n$. Then $f(m) \notin A$, hence $m \notin  \sigma_1(A,B,C)$, in other words, $(\mathcal{M}'_1; \mathbb{N}),(A, B, C)\not\models R(m)$, whence  $(\mathcal{M}'_1; \mathbb{N}),(A, B, C)\not\models \forall xR(x)$ and, by Expansion Property, $\mathcal{M}'_1,(A, B, C)\not\models \forall xR(x)$. Since the $\prec_1$-successor to $\mathbf{v}$ was chosen arbitrarily, it follows, by Lemma \ref{L:neg-and-iff}, that $\mathcal{M}'_1,\mathbf{v}\models \neg\forall xR(x)$.

Next, if $n \in \mathbb{N}$, then consider $f(n) \in \mathbf{v}_2$. Clearly, we have  $(\mathcal{M}'_1; \mathbb{N}),\mathbf{v}\models P(f(n))$. If now $(A,B,C) \in W$ is such that $\mathbf{v}\mathrel{\prec_1}(A,B,C)$ and $(\mathcal{M}'_1; \mathbb{N}),(A, B, C)\models Q(f(n))$, then $f(n) \in A$, thus also $n \in \sigma_1(A,B,C)$ and, therefore, $(\mathcal{M}'_1; \mathbb{N}),(A, B, C)\models R(n)$. We have shown that $(\mathcal{M}'_1; \mathbb{N}),\mathbf{v}\models Q(f(n)) \to R(n)$. Summing up, we get that $(\mathcal{M}'_1; \mathbb{N}),\mathbf{v}\models  P(f(n)) \wedge (Q(f(n)) \to R(n))$ for arbitrary $n \in \mathbb{N}$, therefore also $(\mathcal{M}'_1; \mathbb{N}),\mathbf{v}\models  \forall x\exists y(P(y) \wedge (Q(y) \to R(x)))$ and 
$\mathcal{M}'_1,\mathbf{v}\models  \forall x\exists y(P(y) \wedge (Q(y) \to R(x)))$ by Expansion Property.

($\mathcal{M}'_2,\mathbf{w}\not\models \psi$). Assume that $n \in \mathbb{N}$ and that $(A,B,C) \in W$ is such that $\mathbf{w}\mathrel{\trianglelefteq}(A,B,C)$ and $(\mathcal{M}'_2; \mathbb{N}),(A, B, C)\models P(n)$. Then two cases are possible:

\textit{Case 1}. $\mathbf{w} = (A,B,C)$. Then $n \in \mathbf{w}_1 \cup \mathbf{w}_2 = \mathbf{w}_1$ so that $(\mathcal{M}'_2; \mathbb{N}),(A, B, C)\models Q(n)$, whence also $(\mathcal{M}'_2; \mathbb{N}),(A, B, C)\models Q(n) \vee S$.

\textit{Case 2}. $\mathbf{w}\mathrel{\triangleleft}(A,B,C)$. Then $(\mathcal{M}'_2; \mathbb{N}),(A, B, C)\models S$, therefore also $(\mathcal{M}'_2; \mathbb{N}),(A, B, C)\models Q(n) \vee S$.

Summing up, we have shown that $(\mathcal{M}'_2; \mathbb{N}),\mathbf{w}\models P(n) \to (Q(n) \vee S)$ for arbitrary $n \in \mathbb{N}$, whence it follows that $(\mathcal{M}'_2; \mathbb{N}),\mathbf{w}\models \forall x(P(x) \to (Q(x) \vee S))$ and, after applying Expansion Property, that $\mathcal{M}'_2,\mathbf{w}\models \forall x(P(x) \to (Q(x) \vee S))$. However, we also have $\mathcal{M}'_2,\mathbf{w}\not\models S$ by definition of $\sigma_2$, therefore $\psi$ fails at $(\mathcal{M}'_2,\mathbf{w})$.	
\end{proof}
Next, we define a particular bi-asimulation $\mathbb{A}$ between $\mathcal{M}_1$ and $\mathcal{M}_2$ such that we have $\mathbf{v}\mathrel{\mathbb{A}}\mathbf{w}$. The definition looks as follows:
\begin{definition}\label{D:A-bi-asimulation}
	Relative to the models ${\cal M}_1$ and ${\cal M}_2$ given in
	Definition \ref{D:M}, the relation $\mathbb{A}$ is defined as follows:
	\begin{enumerate}
		\item $\mathbb{A} \subseteq \bigcup_{k \geq 0} [ (W \times \mathbb{N}^{k})
		\times (W \times \mathbb{N}^{k})]$; \item Given any $i,j$, such that $\{i,j\} = \{1,2\}$, any $(A,B,C), (D,E,F)\in W$, any $n \geq 0$ and $\bar{a}_n, \bar{b}_n \in \mathbb{N}^{n}$ we have $(A,B,C)^\frown(\bar{a}_n)\mathrel{\mathbb{A}}(D,E,F)^\frown(\bar{b}_n)$, iff the following conditions hold:
		\begin{enumerate}
			\item Binary relation $[\bar{a}_n \mapsto \bar{b}_n]$ defines a bijection.
			
			\item If $1 \leq k \leq n$ and  $a_k \in A$ then $b_k \in D$;
			
			\item If $1 \leq k \leq n$ and $a_k \in B$, then $b_k \in D \cup
			E$.
		\end{enumerate}
	\end{enumerate}	
\end{definition}
Again, we start by pointing out some easy, but important consequences of our definition:
\begin{corollary}\label{C:contraposed-bi-asimulation}
	For any $i,j$, such that $\{i,j\} = \{1,2\}$, any $(A,B,C), (D,E,F)\in W$, any $n \geq 0$, and any $\bar{a}_n, \bar{b}_n \in \mathbb{N}^{n}$, if we have $(A,B,C)^\frown(\bar{a}_n)\mathrel{\mathbb{A}}(D,E,F)^\frown(\bar{b}_n)$, then, for every $1 \leq k \leq n$, we have:
	 $$
	 b_k \in F \Rightarrow a_k \in C.
	 $$ 
\end{corollary}
\begin{proof}
	Indeed, in the assumptions of the corollary we get the following chain of valid implications:
	\begin{align*}
		 b_k \in F &\Rightarrow b_k \notin D \cup E &&\text{(by $F \cap (D \cup E) = \emptyset$)}\\
		 &\Rightarrow a_k \notin A \cup B &&\text{(by $(A,B,C)^\frown(\bar{a}_n)\mathrel{\mathbb{A}}(D,E,F)^\frown(\bar{b}_n)$)}\\
		 &\Rightarrow a_k \in C&&\text{(by $\{\bar{a}_n\}\subseteq \mathbb{N}$ and $A \cup B \cup C = \mathbb{N}$)}
	\end{align*}
\end{proof}

\begin{corollary}\label{C:A-bi-asimulation}
	For any $i,j$, such that $\{i,j\} = \{1,2\}$, any $(A,B,C), (D,E,F)\in W$, any $n \geq 0$, and any $\bar{a}_n, \bar{b}_n \in \mathbb{N}^{n}$, we have both $(A,B,C)^\frown(\bar{a}_n)\mathrel{\mathbb{A}}(D,E,F)^\frown(\bar{b}_n)$ and $(D,E,F)^\frown(\bar{b}_n)\mathrel{\mathbb{A}}(A,B,C)^\frown(\bar{a}_n)$ iff the following conditions hold:
	\begin{enumerate}
		\item Binary relation $[\bar{a}_n \mapsto \bar{b}_n]$ defines a bijection.
		
		\item For all $1 \leq k \leq n$ it is true that $a_k \in A$ iff $b_k \in D$;
		
		\item For all $1 \leq k \leq n$ it is true that $a_k \in C$ iff $b_k \in
		F$.
	\end{enumerate}
\end{corollary}
\begin{proof}
	The $(\Rightarrow)$-part follows from Corollary \ref{C:contraposed-bi-asimulation} and condition 2(b) of Definition \ref{D:A-bi-asimulation}.
	
	For the $(\Leftarrow)$-part, we observe that if $[\bar{a}_n \mapsto \bar{b}_n]$ defines a bijection, then $[\bar{b}_n \mapsto \bar{a}_n]$ also defines a bijection. The satisfaction of condition 2(b)  of Definition \ref{D:A-bi-asimulation} in both directions is then an immediate consequence of condition 2 of the Corollary, and as for Condition 2(c), we get that, if $1 \leq k \leq n$ and $a_k \in B$, then $a_k \notin C$, whence $b_k \notin E$ by condition 3 of the Corollary. Therefore, $b_k \in D\cup E$. Similarly, if $b_k \in E$, then $b_k \notin F$, whence $a_k \notin C$ by condition 3 of the Corollary. But then $a_k \in A\cup B$ 
\end{proof}
We can now make the crucial step towards our main result:
\begin{lemma}\label{L:A-bi-asimulation}
	The relation $\mathbb{A}$, given in Definition \ref{D:A-bi-asimulation}, is a bi-asimulation between $\mathcal{M}_1$ and $\mathcal{M}_2$, as given in Definition \ref{D:M}, and we have $\mathbf{v}\mathrel{\mathbb{A}}\mathbf{w}$.
\end{lemma}
\begin{proof}
	We check the satisfaction of the conditions given in Definition \ref{D:bi-asimulation} by $\mathbb{A}$. Condition \eqref{E:type} holds by Definition \ref{D:A-bi-asimulation}.1. Condition \eqref{E:atom} for $P$ and $Q$ follows from Definition \ref{D:A-bi-asimulation}.2(b) and \ref{D:A-bi-asimulation}.2(c). It is also clear that for any $i,j$ such that $\{i, j\} = \{1,2\}$, and any $w, v \in W$ we have $w\mathrel{\mathbb{A}}v$, so that $\mathbb{A}$ is non-empty, and, in particular, $\mathbf{v}\mathrel{\mathbb{A}}\mathbf{w}$ holds. We check the remaining conditions in more detail:
	
	\textit{Condition \eqref{E:back}}. Let $i,j$ be such that $\{i, j\} = \{1,2\}$, let $(A, B, C),(D,E,F)\in W$, and let, for some $n \geq 0$, $\bar{a}_n, \bar{b}_n \in \mathbb{N}^{n}$. Assume, moreover, that we have $(A,B,C)^\frown(\bar{a}_n)\mathrel{\mathbb{A}}(D,E,F)^\frown(\bar{b}_n)$, and that some $(G,H,I) \in W$ is such that $(D,E,F)\mathrel{\prec_j}(G,H,I)$ (hence, in particular, $(D,E,F)\trianglelefteq (G,H,I)$). We have to consider two cases, since $B$ can be either empty or infinite.
	
	\textit{Case 1}. $\mid B\mid = \omega$. Then consider the triple $(J,K,L)$ such that:
	\begin{eqnarray*}
		J & = & (A \setminus \{\bar{a}_n\}) \cup [\bar{b}_n \mapsto \bar{a}_n](G);\\
		K & = & (B \setminus \{\bar{a}_n\}) \cup [\bar{b}_n \mapsto \bar{a}_n](H);\\
		L & = & (C \setminus \{\bar{a}_n\}) \cup [\bar{b}_n \mapsto \bar{a}_n](I).
	\end{eqnarray*}
We will show that $(J,K,L)\in W$ and that we have $(A,B,C)\mathrel{\prec_i}(J,K,L)$, $(J,K,L)^\frown(\bar{a}_n)\mathrel{\mathbb{A}}(G,H,I)^\frown(\bar{b}_n)$, and
$(G,H,I)^\frown(\bar{b}_n)\mathrel{\mathbb{A}}(J,K,L)^\frown(\bar{a}_n)$. We break down the demonstration of this statement into the following series of claims:

\textit{Claim 1}. $J \cup K \cup L = \mathbb{N}$.

Indeed, we know that
\begin{align*}
[\bar{b}_n \mapsto \bar{a}_n](G) \cup [\bar{b}_n \mapsto \bar{a}_n](H) \cup [\bar{b}_n \mapsto \bar{a}_n](I) &= [\bar{b}_n \mapsto \bar{a}_n](G \cup H \cup I)\\
&=[\bar{b}_n \mapsto \bar{a}_n](\mathbb{N}) = \{\bar{a}_n\}
\end{align*}
since $(G,H,I)\in W$ and, therefore, $G \cup H \cup I = \mathbb{N}$.
On the other hand, we know that:
\begin{align*}
(A \setminus \{\bar{a}_n\})\cup (B \setminus \{\bar{a}_n\})\cup (C \setminus \{\bar{a}_n\}) &= (A \cup B \cup C)\setminus \{\bar{a}_n\}\\
&= \mathbb{N}\setminus \{\bar{a}_n\}
\end{align*}
since $(A,B,C)\in W$ and, therefore, $A \cup B \cup C = \mathbb{N}$.
Adding the two equalities together, we get that:
\begin{align*}
J \cup K \cup L &= (\mathbb{N}\setminus \{\bar{a}_n\})\cup \{\bar{a}_n\} = \mathbb{N},
\end{align*}
as desired.

\textit{Claim 2}. $J$, $K$, and $L$ are infinite.

Indeed, $A$, $B$, and $C$ are infinite, $\{\bar{a}_n\}$ is finite, and the following inclusions hold:
\begin{align*}
&A \setminus \{\bar{a}_n\} \subseteq J,\,B \setminus \{\bar{a}_n\} \subseteq K,\,C \setminus \{\bar{a}_n\} \subseteq L.
\end{align*}

\textit{Claim 3}. $J$, $K$, and $L$ are pairwise disjoint.

Indeed, take $J$ and $K$, for example. By definition, we have that $J \cap K$ is equal to  
$$
((A \setminus \{\bar{a}_n\}) \cup [\bar{b}_n \mapsto \bar{a}_n](G)) \cap ((B \setminus \{\bar{a}_n\}) \cup [\bar{b}_n \mapsto \bar{a}_n](H)),
$$
By application of distributivity laws, we further get that
\begin{align*}
J \cap K &= ((A \setminus \{\bar{a}_n\}) \cap (B \setminus \{\bar{a}_n\}))\cup\\
&\qquad\qquad\cup ((A \setminus \{\bar{a}_n\}) \cap [\bar{b}_n \mapsto \bar{a}_n](H))\\
&\qquad\qquad\cup ((B \setminus \{\bar{a}_n\}) \cap [\bar{b}_n \mapsto \bar{a}_n](G))\\
&\qquad\qquad\cup ([\bar{b}_n \mapsto \bar{a}_n](G) \cap [\bar{b}_n \mapsto \bar{a}_n](H))
\end{align*}  
Next, we know that $(A,B,C)\in W$, therefore, $A \cap B = \emptyset$, which implies that $(A \setminus \{\bar{a}_n\})\cap (B \setminus \{\bar{a}_n\}) = \emptyset$.

Moreover, we have 
$$
(A \setminus \{\bar{a}_n\}) \cap [\bar{b}_n \mapsto \bar{a}_n](H) \subseteq (A \setminus \{\bar{a}_n\}) \cap \{\bar{a}_n\} = \emptyset.
$$
By a parallel argument, we can see that also $(B \setminus \{\bar{a}_n\}) \cap [\bar{b}_n \mapsto \bar{a}_n](G) = \emptyset$.

Finally, we observe that $\bar{b}_n \mapsto \bar{a}_n$ is a bijection, and $(G,H,I)\in W$, therefore, $G \cap H = \emptyset$, whence it follows that $[\bar{b}_n \mapsto \bar{a}_n](G) \cap [\bar{b}_n \mapsto \bar{a}_n](H) = \emptyset$.

Summing up, we see that all the four sets in the union defining $J \cap K$ are empty and that we have that $J \cap K = \emptyset$. The other cases are similar.

\textit{Claim 4}. $(J,K,L)\in W$.

By Claims 1--3.

\textit{Claim 5}. $(A,B,C)\mathrel{\trianglelefteq}(J,K,L)$.

Indeed, if $a \in A$, and $a \notin \{\bar{a}_n\}$, then $a \in J$ by definition of $J$. Otherwise we have both $a \in A$ and
$a = a_k$ for some $1 \leq k \leq n$, but then $b_k \in D$ by $(A,B,C)^\frown(\bar{a}_n)\mathrel{\mathbb{A}}(D,E,F)^\frown(\bar{b}_n)$. Next, note that $(D,E,F)\trianglelefteq (G,H,I)$ implies that $D \subseteq G$, which means that $b_k \in G$. But the latter means that $a = a_k \in [\bar{b}_n \mapsto \bar{a}_n](G) \subseteq J$. Since $a \in A$ was chosen arbitrarily, we have shown that $A \subseteq J$.

Next, if $a \in L$, then either $a \in C \setminus \{\bar{a}_n\} \subseteq C$, or $a \in [\bar{b}_n \mapsto \bar{a}_n](I)$. In the latter case, $a = a_k$ for some $1 \leq k \leq n$, and also $b_k \in I$. Since $(D,E,F)\trianglelefteq (G,H,I)$ implies that $I \subseteq
F$, we get that $b_k \in F$, but the latter means, by $(A,B,C)^\frown(\bar{a}_n)\mathrel{\mathbb{A}}(D,E,F)^\frown(\bar{b}_n)$ and Corollary \ref{C:contraposed-bi-asimulation}, that $a_k \in C$. Since $a \in L$ was chosen arbitrarily, we have shown that $L \subseteq C$.

\textit{Claim 6}. $(\mathbf{v}
\trianglelefteq (A,B,C)\,\&\,\mid B \cap \mathbf{v}_2\mid = \omega) \Rightarrow \mid K \cap \mathbf{v}_2\mid = \omega$.

We observe that if  $B \cap \mathbf{v}_2$ is infinite, then so is $(B
\setminus \{\bar{a}_n\}) \cap \mathbf{v}_2$, given that $\{\bar{a}_n\}$
is finite. But since, according to our definition of $K$ we have
$(B \setminus \{\bar{a}_n\}) \subseteq K$, $K \cap \mathbf{v}_2$
must be infinite, too.

\textit{Claim 7}. $(A,B,C)\mathrel{\prec_i}(J,K,L)$.

By Claims 5 and 6.

\textit{Claim 8}. $(J,K,L)^\frown(\bar{a}_n)\mathrel{\mathbb{A}}(G,H,I)^\frown(\bar{b}_n)$, and
$(G,H,I)^\frown(\bar{b}_n)\mathrel{\mathbb{A}}(J,K,L)^\frown(\bar{a}_n)$.

Indeed, for any given $1 \leq k \leq
n$, we have $a_k \in J$ iff $a_k \in [\bar{b}_n \mapsto \bar{a}_n](G)$ iff $b_k \in G$, since $\bar{b}_n \mapsto \bar{a}_n$ defines a bijection. Similarly, we have $a_k \in K$ iff  $a_k \in [\bar{b}_n \mapsto \bar{a}_n](I)$ iff $b_k \in I$. But then our Claim follows from Corollary \ref{C:A-bi-asimulation}.

The correctness of our construction for Case 1 of condition \eqref{E:back} now follows from Claims 4, 7, and 8.

\textit{Case 2}. $B = \emptyset$. In this case, we partition $C \setminus \{\bar{a}_n\}$ into two
disjoint infinite sets $C_1$ and $C_2$, and define $(J,K,L)$
as follows:

\begin{eqnarray*}
	J & = & (A \setminus \{\bar{a}_n\}) \cup [\bar{b}_n \mapsto \bar{a}_n](G) ;\\
	K & = & C_1 \cup [\bar{b}_n \mapsto \bar{a}_n](H);\\
	L & = & C_2 \cup [\bar{b}_n \mapsto \bar{a}_n](I).
\end{eqnarray*}

We will demonstrate that all the claims that we have made in the previous case still hold.

\textit{Claim 1}. $J \cup K \cup L = \mathbb{N}$.

Again, we know that
\begin{align*}
	[\bar{b}_n \mapsto \bar{a}_n](G) \cup [\bar{b}_n \mapsto \bar{a}_n](H) \cup [\bar{b}_n \mapsto \bar{a}_n](I) &= [\bar{b}_n \mapsto \bar{a}_n](G \cup H \cup I)\\
	&=[\bar{b}_n \mapsto \bar{a}_n](\mathbb{N}) = \{\bar{a}_n\}
\end{align*}
since $(G,H,I)\in W$ and, therefore, $G \cup H \cup I = \mathbb{N}$.
On the other hand, we know that:
\begin{align*}
	(A \setminus \{\bar{a}_n\})\cup C_1 \cup C_2 &= (A \setminus \{\bar{a}_n\}) \cup (C \setminus \{\bar{a}_n\})\\
	&= (A \cup C)\setminus \{\bar{a}_n\} = \mathbb{N}\setminus \{\bar{a}_n\}
\end{align*}
since $(A,B,C)\in W$ and $B = \emptyset$, therefore, $A  \cup C = \mathbb{N}$.
Adding the two equalities together, we get that:
\begin{align*}
	J \cup K \cup L &= (\mathbb{N}\setminus \{\bar{a}_n\})\cup \{\bar{a}_n\} = \mathbb{N},
\end{align*}
as desired.

\textit{Claim 2}. $J$, $K$, and $L$ are infinite.

Indeed, $A$, $C_1$ and $C_2$ are infinite, $\{\bar{a}_n\}$ is finite, and the following inclusions hold:
\begin{align*}
	&A \setminus \{\bar{a}_n\} \subseteq J,\,C_1 \subseteq K,\,C_2 \subseteq L.
\end{align*}

\textit{Claim 3}. $J$, $K$, and $L$ are pairwise disjoint.

Indeed, take $J$ and $K$, for example. By definition, we have that
\begin{align*}
	J \cap K &= ((A \setminus \{\bar{a}_n\}) \cup [\bar{b}_n \mapsto \bar{a}_n](G)) \cap (C_1 \cup [\bar{b}_n \mapsto \bar{a}_n](H))\\
	&\subseteq ((A \setminus \{\bar{a}_n\}) \cup [\bar{b}_n \mapsto \bar{a}_n](G)) \cap ((C \setminus \{\bar{a}_n\})\cup [\bar{b}_n \mapsto \bar{a}_n](H))
\end{align*}

By application of distributivity laws, we further get that
\begin{align*}
	J \cap K &\subseteq ((A \setminus \{\bar{a}_n\}) \cap (C \setminus \{\bar{a}_n\}))\cup\\
	&\qquad\qquad\cup ((A \setminus \{\bar{a}_n\}) \cap [\bar{b}_n \mapsto \bar{a}_n](H))\\
	&\qquad\qquad\cup ((C \setminus \{\bar{a}_n\}) \cap [\bar{b}_n \mapsto \bar{a}_n](G))\\
	&\qquad\qquad\cup ([\bar{b}_n \mapsto \bar{a}_n](G) \cap [\bar{b}_n \mapsto \bar{a}_n](H))
\end{align*}  
Next, we know that $(A,B,C)\in W$, therefore, $A \cap C = \emptyset$, which implies that $(A \setminus \{\bar{a}_n\})\cap (C \setminus \{\bar{a}_n\}) = \emptyset$.

Moreover, we have 
$$
(A \setminus \{\bar{a}_n\}) \cap [\bar{b}_n \mapsto \bar{a}_n](H) \subseteq (A \setminus \{\bar{a}_n\}) \cap \{\bar{a}_n\} = \emptyset.
$$
By a parallel argument, we can see that also $(C \setminus \{\bar{a}_n\}) \cap [\bar{b}_n \mapsto \bar{a}_n](G) = \emptyset$.

Finally, we observe that $\bar{b}_n \mapsto \bar{a}_n$ is a bijection, and $(G,H,I)\in W$, therefore, $G \cap H = \emptyset$, whence it follows that $[\bar{b}_n \mapsto \bar{a}_n](G) \cap [\bar{b}_n \mapsto \bar{a}_n](H) = \emptyset$.

Summing up, we see that all the four sets in the union extending $J \cap K$ are empty and that we have that $J \cap K = \emptyset$. The other cases are similar.

\textit{Claim 4}. $(J,K,L)\in W$.

By Claims 1--3.

\textit{Claim 5}. $(A,B,C)\mathrel{\trianglelefteq}(J,K,L)$.

The argument for $A \subseteq J$ is exactly the same as in Case 1. To see that $L \subseteq C$, note that $L = C_2 \cup [\bar{b}_n \mapsto \bar{a}_n](I) \subseteq (C \setminus \{\bar{a}_n\}) \cup [\bar{b}_n \mapsto \bar{a}_n](I)$, so that the argument used for Case 1 is applicable also here.

\textit{Claim 6}. $(\mathbf{v}
\trianglelefteq (A,B,C)\,\&\,\mid B \cap \mathbf{v}_2\mid = \omega) \Rightarrow \mid K \cap \mathbf{v}_2\mid = \omega$.

The claim holds trivially since $\mid B \cap \mathbf{v}_2\mid = \omega$ is falsified by the assumption of Case 2.

\textit{Claim 7}. $(A,B,C)\mathrel{\prec_i}(J,K,L)$.

By Claims 5 and 6.

\textit{Claim 8}. $(J,K,L)^\frown(\bar{a}_n)\mathrel{\mathbb{A}}(G,H,I)^\frown(\bar{b}_n)$, and
$(G,H,I)^\frown(\bar{b}_n)\mathrel{\mathbb{A}}(J,K,L)^\frown(\bar{a}_n)$.

Again the claim follows by the argument used for Claim 8 of Case 1.

\textit{Condition \eqref{E:forth}}. Let $i,j$ be such that $\{i, j\} = \{1,2\}$, let $(A, B, C), (D,E,F)\in W$, and let, for some $n \geq 0$, $\bar{a}_n, \bar{b}_n \in \mathbb{N}^{n}$. Assume, moreover, that we have $(A,B,C)^\frown(\bar{a}_n)\mathrel{\mathbb{A}}(D,E,F)^\frown(\bar{b}_n)$, and that some $(J,K,L) \in W$ is such that $(J,K,L)\mathrel{\prec_i}(A,B,C)$ (hence, in particular, $(A,B,C)\mathrel{\trianglerighteq}(J,K,L)$). Again, we have to consider two cases, since $E$ can be either empty or infinite.

\textit{Case 1}. $\mid E\mid = \omega$. Then consider the triple $(G,H,I)$ such that:
\begin{eqnarray*}
	G & = & (D \setminus \{\bar{b}_n\}) \cup [\bar{a}_n \mapsto \bar{b}_n](J);\\
	H & = & (E \setminus \{\bar{b}_n\}) \cup [\bar{a}_n \mapsto \bar{b}_n](K);\\
	I & = & (F \setminus \{\bar{b}_n\}) \cup [\bar{a}_n \mapsto \bar{b}_n](L).
\end{eqnarray*}

We will show that $(G,H,I)\in W$ and that we have $(G,H,I)\mathrel{\prec_j}(D,E,F)$, $(J,K,L)^\frown(\bar{a}_n)\mathrel{\mathbb{A}}(G,H,I)^\frown(\bar{b}_n)$, and
$(G,H,I)^\frown(\bar{b}_n)\mathrel{\mathbb{A}}(J,K,L)^\frown(\bar{a}_n)$. We break down the demonstration of this statement into our usual series of eight claims. The arguments, for the most part, just dualize the arguments given for Case 1 of condition \eqref{E:back}, but we spell them out nonetheless.

\textit{Claim 1}. $G \cup H \cup I = \mathbb{N}$.

Indeed, we know that
\begin{align*}
	[\bar{a}_n \mapsto \bar{b}_n](J) \cup [\bar{a}_n \mapsto \bar{b}_n](K) \cup [\bar{a}_n \mapsto \bar{b}_n](L) &= [\bar{a}_n \mapsto \bar{b}_n](J \cup K \cup L)\\
	&=[\bar{a}_n \mapsto \bar{b}_n](\mathbb{N}) = \{\bar{b}_n\}
\end{align*}
since $(J,K,L)\in W$ and, therefore, $J \cup K \cup L = \mathbb{N}$.
On the other hand, we know that:
\begin{align*}
	(D \setminus \{\bar{b}_n\})\cup (E \setminus \{\bar{b}_n\})\cup (F \setminus \{\bar{b}_n\}) &= (D \cup E \cup F)\setminus \{\bar{b}_n\}\\
	&= \mathbb{N}\setminus \{\bar{b}_n\}
\end{align*}
since $(D,E,F)\in W$ and, therefore, $D \cup E \cup F = \mathbb{N}$.
Adding the two equalities together, we get that:
\begin{align*}
	G \cup H \cup I &= (\mathbb{N}\setminus \{\bar{b}_n\})\cup \{\bar{b}_n\} = \mathbb{N},
\end{align*}
as desired.

\textit{Claim 2}. $G$, $H$, and $I$ are infinite.

Indeed, $D$, $E$, and $F$ are infinite, $\{\bar{b}_n\}$ is finite, and the following inclusions hold:
\begin{align*}
	&D \setminus \{\bar{b}_n\} \subseteq G,\,E \setminus \{\bar{b}_n\} \subseteq H ,\,F \setminus \{\bar{b}_n\} \subseteq I.
\end{align*}

\textit{Claim 3}. $G$, $H$, and $I$ are pairwise disjoint.

Indeed, take $G$ and $H$, for example. By definition, we have that $G \cap H$ is equal to  
$$
((D \setminus \{\bar{b}_n\}) \cup [\bar{a}_n \mapsto \bar{b}_n](J)) \cap ((E \setminus \{\bar{b}_n\}) \cup [\bar{a}_n \mapsto \bar{b}_n](K)),
$$
By application of distributivity laws, we further get that
\begin{align*}
	G \cap H &= ((D \setminus \{\bar{b}_n\}) \cap (E \setminus \{\bar{b}_n\}))\cup\\
	&\qquad\qquad\cup ((D \setminus \{\bar{b}_n\}) \cap [\bar{a}_n \mapsto \bar{b}_n](K))\\
	&\qquad\qquad\cup ((E \setminus \{\bar{b}_n\}) \cap [\bar{a}_n \mapsto \bar{b}_n](J))\\
	&\qquad\qquad\cup ([\bar{a}_n \mapsto \bar{b}_n](J) \cap [\bar{a}_n \mapsto \bar{b}_n](K))
\end{align*}  
Next, we know that $(D,E,F)\in W$, therefore, $D \cap E = \emptyset$, which implies that $(D \setminus \{\bar{b}_n\})\cap (E \setminus \{\bar{b}_n\}) = \emptyset$.

Moreover, we have 
$$
(D \setminus \{\bar{b}_n\}) \cap [\bar{a}_n \mapsto \bar{b}_n](K) \subseteq (D \setminus \{\bar{b}_n\}) \cap \{\bar{b}_n\} = \emptyset.
$$
By a parallel argument, we can see that also $(E \setminus \{\bar{b}_n\}) \cap [\bar{a}_n \mapsto \bar{b}_n](J) = \emptyset$.

Finally, we observe that $\bar{a}_n \mapsto \bar{b}_n$ is a bijection, and $(J,K,L)\in W$, therefore, $J \cap K = \emptyset$, whence it follows that $[\bar{a}_n \mapsto \bar{b}_n](J) \cap [\bar{a}_n \mapsto \bar{b}_n](K) = \emptyset$.

Summing up, we see that all the four sets in the union defining $G \cap H$ are empty and that we have that $G \cap H = \emptyset$. The other cases are similar.

\textit{Claim 4}. $(G,H,I)\in W$.

By Claims 1--3.

\textit{Claim 5}. $(G,H,I)\mathrel{\trianglelefteq}(D,E,F)$.

Indeed, if $b \in G$, then either $b \in D \setminus \{\bar{b}_n\} \subseteq D$, or $b \in [\bar{a}_n \mapsto \bar{b}_n](J)$. In the latter case, $b = b_k$ for some $1 \leq k \leq n$, and also $a_k \in J$. Since $(J,K,L)\trianglelefteq (A,B,C)$ implies that $J \subseteq
A$, we get that $a_k \in A$, but the latter means, by $(A,B,C)^\frown(\bar{a}_n)\mathrel{\mathbb{A}}(D,E,F)^\frown(\bar{b}_n)$, that $b_k \in D$. Since $b \in G$ was chosen arbitrarily, we have shown that $G \subseteq D$.

Next, if $b \in F$, and $b \notin \{\bar{b}_n\}$, then $b \in I$ by definition of $F$. Otherwise we have both $b \in F$ and
$b = b_k$ for some $1 \leq k \leq n$, but then $a_k \in C$ by $(A,B,C)^\frown(\bar{a}_n)\mathrel{\mathbb{A}}(D,E,F)^\frown(\bar{b}_n)$ and Corollary \ref{C:contraposed-bi-asimulation}. Now, note that $(J,K,L)\mathrel{\trianglelefteq}(A,B,C)$ implies that $C \subseteq L$, which means that $a_k \in L$. But the latter means that $b = b_k \in [\bar{a}_n \mapsto \bar{b}_n](L) \subseteq I$. Since $b \in F$ was chosen arbitrarily, we have shown that $F \subseteq I$.

\textit{Claim 6}. $(\mathbf{v}
\trianglelefteq (G,H,I)\,\&\,\mid H \cap \mathbf{v}_2\mid = \omega) \Rightarrow \mid E \cap \mathbf{v}_2\mid = \omega$.

We observe that if  $H \cap \mathbf{v}_2$ is infinite, then so is $(E
\setminus \{\bar{b}_n\}) \cap \mathbf{v}_2$, given that $H = (E \setminus \{\bar{b}_n\}) \cup [\bar{a}_n \mapsto \bar{b}_n](K)$ and $[\bar{a}_n \mapsto \bar{b}_n](K)$
is finite. But then $E \cap \mathbf{v}_2$ 
must be infinite, too.

\textit{Claim 7}. $(G,H,I)\mathrel{\prec_j}(D,E,F)$.

By Claims 5 and 6.

\textit{Claim 8}. $(J,K,L)^\frown(\bar{a}_n)\mathrel{\mathbb{A}}(G,H,I)^\frown(\bar{b}_n)$, and
$(G,H,I)^\frown(\bar{b}_n)\mathrel{\mathbb{A}}(J,K,L)^\frown(\bar{a}_n)$.

Indeed, for any given $1 \leq k \leq
n$, we have $b_k \in G$ iff $b_k \in [\bar{a}_n \mapsto \bar{b}_n](J)$ iff $a_k \in J$, since $\bar{a}_n \mapsto \bar{b}_n$ defines a bijection. Similarly, we have $b_k \in I$ iff  $b_k \in [\bar{a}_n \mapsto \bar{b}_n](L)$ iff $a_k \in L$. But then our Claim follows from Corollary \ref{C:A-bi-asimulation}.

The correctness of our construction for Case 1 of condition \eqref{E:forth} now follows from Claims 4, 7, and 8.

\textit{Case 2}. $E = \emptyset$. In this case, we partition $D \setminus \{\bar{b}_n\}$ into two
disjoint infinite sets $D_1$ and $D_2$. Additionally, in case both $D \cap \mathbf{v}_1$ and $D \cap \mathbf{v}_2$ are infinite (which means that $(D \setminus \{\bar{b}_n\}) \cap \mathbf{v}_2$ and $(D \setminus \{\bar{b}_n\}) \cap \mathbf{v}_1$ are infinite, too), we ensure that we have $(D \setminus \{\bar{b}_n\}) \cap \mathbf{v}_2 \subseteq D_1$ and $(D \setminus \{\bar{b}_n\}) \cap \mathbf{v}_1 \subseteq D_2$. Then we define $(G,H,I)$
as follows:

\begin{eqnarray*}
	G & = & D_1 \cup [\bar{a}_n \mapsto \bar{b}_n](J);\\
	H & = & D_2 \cup [\bar{a}_n \mapsto \bar{b}_n](K);\\
	I & = & (F \setminus \{\bar{b}_n\}) \cup [\bar{a}_n \mapsto \bar{b}_n](L).
\end{eqnarray*}
We will demonstrate our usual eight claims in some detail again, even though the arguments mostly dualize the proofs given for the respective claims in Case 2 of condition \eqref{E:back} (Claim 6 below being, perhaps, the only exception to this rule).

\textit{Claim 1}. $G \cup H \cup I = \mathbb{N}$.

Again, we know that
\begin{align*}
	[\bar{a}_n \mapsto \bar{b}_n](J) \cup [\bar{a}_n \mapsto \bar{b}_n](K) \cup [\bar{a}_n \mapsto \bar{b}_n](L) &= [\bar{a}_n \mapsto \bar{b}_n](J \cup K \cup L)\\
	&=[\bar{a}_n \mapsto \bar{b}_n](\mathbb{N}) = \{\bar{b}_n\}
\end{align*}
since $(J,K,L)\in W$ and, therefore, $J \cup K \cup L = \mathbb{N}$.
On the other hand, we know that:
\begin{align*}
	(F \setminus \{\bar{b}_n\})\cup D_1 \cup D_2 &= (F \setminus \{\bar{b}_n\}) \cup (D \setminus \{\bar{b}_n\})\\
	&= (F \cup D)\setminus \{\bar{a}_n\} = \mathbb{N}\setminus \{\bar{a}_n\}
\end{align*}
since $(D,E,F)\in W$ and $E = \emptyset$, therefore, $D  \cup F = \mathbb{N}$.
Adding the two equalities together, we get that:
\begin{align*}
	G \cup H \cup I &= (\mathbb{N}\setminus \{\bar{b}_n\})\cup \{\bar{b}_n\} = \mathbb{N},
\end{align*}
as desired.

\textit{Claim 2}. $G$, $H$, and $I$ are infinite.

Indeed, $D_1$, $D_2$, and $F$ are infinite, $\{\bar{b}_n\}$ is finite, and the following inclusions hold:
\begin{align*}
	&D_1 \subseteq G,\,D_2 \subseteq H,\,F \setminus \{\bar{b}_n\} \subseteq I.
\end{align*}

\textit{Claim 3}. $G$, $H$, and $I$ are pairwise disjoint.

Indeed, take $G$ and $I$, for example. By definition, we have that
\begin{align*}
	G \cap I &= ((F \setminus \{\bar{b}_n\}) \cup [\bar{a}_n \mapsto \bar{b}_n](L)) \cap (D_1 \cup [\bar{a}_n \mapsto \bar{b}_n](J))\\
	&\subseteq ((F \setminus \{\bar{b}_n\}) \cup [\bar{a}_n \mapsto \bar{b}_n](L)) \cap ((D \setminus \{\bar{b}_n\})\cup [\bar{a}_n \mapsto \bar{b}_n](J))
\end{align*}

By application of distributivity laws, we further get that
\begin{align*}
	G \cap I &\subseteq ((F \setminus \{\bar{b}_n\}) \cap (D \setminus \{\bar{b}_n\}))\cup\\
	&\qquad\qquad\cup ((F \setminus \{\bar{b}_n\}) \cap [\bar{a}_n \mapsto \bar{b}_n](J))\\
	&\qquad\qquad\cup ((D \setminus \{\bar{b}_n\}) \cap [\bar{a}_n \mapsto \bar{b}_n](L))\\
	&\qquad\qquad\cup ([\bar{a}_n \mapsto \bar{b}_n](J) \cap [\bar{a}_n \mapsto \bar{b}_n](L))
\end{align*}  
Next, we know that $(D,E,F)\in W$, therefore, $D \cap F = \emptyset$, which implies that $(D \setminus \{\bar{b}_n\})\cap (F \setminus \{\bar{b}_n\}) = \emptyset$.

Moreover, we have 
$$
(F \setminus \{\bar{b}_n\}) \cap [\bar{a}_n \mapsto \bar{b}_n](J) \subseteq (F \setminus \{\bar{b}_n\}) \cap \{\bar{b}_n\} = \emptyset.
$$
By a parallel argument, we can see that also $(D \setminus \{\bar{b}_n\}) \cap [\bar{a}_n \mapsto \bar{b}_n](L) = \emptyset$.

Finally, we observe that $\bar{a}_n \mapsto \bar{b}_n$ is a bijection, and $(J,K,L)\in W$, therefore, $J \cap L = \emptyset$, whence it follows that $[\bar{a}_n \mapsto \bar{b}_n](J) \cap [\bar{a}_n \mapsto \bar{b}_n](L) = \emptyset$.

Summing up, we see that all the four sets in the union extending $G \cap I$ are empty and that we have $G \cap I = \emptyset$. The other cases are similar.

\textit{Claim 4}. $(G,H,I)\in W$.

By Claims 1--3.

\textit{Claim 5}. $(G,H,I)\mathrel{\trianglelefteq}(D,E,F)$.

The argument for $F \subseteq I$ is exactly the same as in Case 1. To see that $G \subseteq D$, note that $G = D_1 \cup [\bar{a}_n \mapsto \bar{b}_n](J) \subseteq (D \setminus \{\bar{b}_n\}) \cup [\bar{a}_n \mapsto \bar{b}_n](J)$, so that the argument used for Case 1 is applicable also here.

\textit{Claim 6}. $(\mathbf{v}
\mathrel{\trianglelefteq}(G,H,I)\,\&\,\mid H \cap \mathbf{v}_2\mid = \omega) \Rightarrow \mid E \cap \mathbf{v}_2\mid = \omega$.

Indeed, if $\mathbf{v}
\mathrel{\trianglelefteq}(G,H,I)$, then, in particular, $\mathbf{v}_1 \subseteq G$, so that $\mid G \cap \mathbf{v}_1\mid = \omega$. But we have $G = D_1 \cup [\bar{a}_n \mapsto \bar{b}_n](J)$ and $[\bar{a}_n \mapsto \bar{b}_n](J)$ is finite; therefore, both $\mathbf{v}_1 \cap D_1$, and its superset,  $\mathbf{v}_1 \cap D$, must be infinite.

If also $\mid H \cap \mathbf{v}_2\mid = \omega$, then a parallel argument shows that also $\mathbf{v}_2 \cap D_2$ and $\mathbf{v}_2 \cap D$ are infinite.

But, since both $\mathbf{v}_1 \cap D$ and $\mathbf{v}_2 \cap D$ are infinite, we must have, by the choice of $D_1$ and $D_2$, that $\mathbf{v}_2 \cap (D \setminus \{\bar{b}_n\}) \subseteq D_1$. On the other hand, $D_2 \subseteq (D \setminus \{\bar{b}_n\})$ implies that $D_2 = D_2 \cap (D \setminus \{\bar{b}_n\})$.

Therefore:
$$
\mathbf{v}_2 \cap D_2 = \mathbf{v}_2 \cap (D \setminus \{\bar{b}_n\})\cap D_2 \subseteq D_1 \cap D_2 = \emptyset. 
$$  
Since we have $H = D_2 \cup [\bar{a}_n \mapsto \bar{b}_n](K)$, and $[\bar{a}_n \mapsto \bar{b}_n](K)$ is clearly finite, the set $H \cap \mathbf{v}_2$ can be at most finite, which is a contradiction.

\textit{Claim 7}. $(A,B,C)\mathrel{\prec_i}(J,K,L)$.

By Claims 5 and 6.

\textit{Claim 8}. $(J,K,L)^\frown(\bar{a}_n)\mathrel{\mathbb{A}}(G,H,I)^\frown(\bar{b}_n)$, and
$(G,H,I)^\frown(\bar{b}_n)\mathrel{\mathbb{A}}(J,K,L)^\frown(\bar{a}_n)$.

Again the claim follows by the argument used for Claim 8 of Case 1. 

\textit{Condition} \eqref{E:left}. Let $i,j$ be such that $\{i, j\} = \{1,2\}$, let $(A, B, C), (D,E,F)\in W$, and let, for some $n \geq 0$, $\bar{a}_n, \bar{b}_n \in \mathbb{N}^{n}$. Assume, moreover, that we have $(A,B,C)^\frown(\bar{a}_n)\mathrel{\mathbb{A}}(D,E,F)^\frown(\bar{b}_n)$, and that $b \in \mathbb{N}$. If $b = b_k$ for some $1 \leq k \leq n$, then we set $a: = a_k$. Otherwise, given that $b \notin \{\bar{b}_n\}$, we choose any $a \in C \setminus \{\bar{a}_n\}$. We can do this since $C$ is infinite and $\{\bar{a}_n\}$ is finite. In both cases we get that $(A,B,C)^\frown(\bar{a}_n)^\frown(a)\mathrel{\mathbb{A}}(D,E,F)^\frown(\bar{b}_n)^\frown(b)$ by definition of $\mathbb{A}$.

\textit{Condition} \eqref{E:right}. Let $i,j$ be such that $\{i, j\} = \{1,2\}$, let $(A, B, C), (D,E,F)\in W$, and let, for some $n \geq 0$, $\bar{a}_n, \bar{b}_n \in \mathbb{N}^{n}$. Assume, moreover, that we have $(A,B,C)^\frown(\bar{a}_n)\mathrel{\mathbb{A}}(D,E,F)^\frown(\bar{b}_n)$, and that $a \in \mathbb{N}$. If $a = a_k$ for some $1 \leq k \leq n$, then we set $b: = b_k$. Otherwise, given that $a \notin \{\bar{a}_n\}$, we choose any $b \in D \setminus \{\bar{b}_n\}$. We can do this since $D$ is infinite and $\{\bar{b}_n\}$ is finite. In both cases we get that $(A,B,C)^\frown(\bar{a}_n)^\frown(a)\mathrel{\mathbb{A}}(D,E,F)^\frown(\bar{b}_n)^\frown(b)$ by definition of $\mathbb{A}$.
\end{proof}
At this point, it only remains to reap the fruits of the tedious work towards the host of previous lemmas and corollaries:
\begin{theorem}\label{T:main}
	$\mathsf{QBIL}$ fails CIP.
\end{theorem}
\begin{proof}
	Indeed, we have $\models \phi \to \psi$ and $\Theta_\phi \cap \Theta_\psi = \Sigma$. If now $\theta \in L_\emptyset(\Sigma)$ is such that both $\models \phi \to \theta$ and $\models \theta \to \psi$, then, by Lemma \ref{L:extendable}, we must have both $\mathcal{M}'_1,\mathbf{v}\models \theta$ and $\mathcal{M}'_2,\mathbf{w}\not\models \theta$. We also have $\mathcal{M}_i = \mathcal{M}'_i\upharpoonright\Sigma$ for $i \in \{1,2\}$, therefore, by Expansion Property, we get that $\mathcal{M}_1,\mathbf{v}\models \theta$ and $\mathcal{M}_2,\mathbf{w}\not\models \theta$.  On the other hand, we know, by Lemma \ref{L:A-bi-asimulation}, that the relation $\mathbb{A}$, given in Definition \ref{D:A-bi-asimulation}, is a bi-asimulation between $\mathcal{M}_1$ and $\mathcal{M}_2$ and that we have $\mathbf{v}\mathrel{\mathbb{A}}\mathbf{w}$. Therefore, $\mathcal{M}_1,\mathbf{v}\models \theta$ implies, by Lemma \ref{L:preserve}, that $\mathcal{M}_2,\mathbf{w}\models \theta$. The obtained contradiction shows that no interpolant exists for $\phi \to \psi$. 
\end{proof}
\textbf{Remark 2}. Note that Corollary \ref{C:M1}.1--2 implies that the submodels generated by $\mathbf{v}$ and $\mathbf{w}$ in $\mathcal{M}_1$ and $\mathcal{M}_2$, respectively, are exactly the models $\mathcal{M}_1$ and $\mathcal{M}_2$ defined in \cite[\S7]{mou}. Moreover, Definition \ref{D:A-bi-asimulation} and the proof of Lemma \ref{L:A-bi-asimulation}, except for the part treating condition \eqref{E:forth} of Definition \ref{D:bi-asimulation}, coincide with the respective constructions of \cite[\S7]{mou}. Thus our main construction in this paper is, in a very precise sense, just an extension of the main construction given in \cite[\S7]{mou}.

\section{Conclusion}\label{con}
In this article we have refuted the Craig Interpolation Property for predicate bi-intuitionistic logic, showing how different the situation is from the propositional case that was solved positively in \cite{ko}. We proved that Mints's counterexample \cite{mou} for predicate intuitionitic logic with constant domains also did the work in the present context.
It is clear that, although we allowed constants in our presentation, their role is purely technical, and that due to the nature of Mints's counterexample, we have also disproved CIP for the purely relational variant of $\mathsf{QBIL}$. The present work still leaves some   related open questions, such as  the status of the Beth definability property in the bi-intuitionistic  setting. For the time being, we have  left these and other  similar matters as topics for future research.

\section{Acknowledgements}
Grigory Olkhovikov is supported by Deutsche
Forschungsgemeinschaft (DFG), project OL 589/1-1. Guillermo Badia is partially supported by  the Australian Research Council grant DE220100544.

Grigory Olkhovikov would like to dedicate this paper to the memory of his dear co-author, late Professor Grigori Mints, whose strong and incisive mind has been missing in this world since May 29, 2014, and has been deeply missed on every single day since that fateful date. Guillermo Badia most emphatically supports this decision.

}
\end{document}